\documentclass{article}
\usepackage[a4paper, total={7in, 10in}]{geometry}
\usepackage{amsmath,amsthm,amssymb,mathrsfs,amstext, titlesec,enumitem, comment, graphicx, color, xcolor, stmaryrd,mathabx}
\usepackage{hyperref}
\usepackage{xpatch}
\usepackage{tikz-cd}
\makeatletter
\def\Ddots{\mathinner{\mkern1mu\raise\p@
\vbox{\kern7\p@\hbox{.}}\mkern2mu
\raise4\p@\hbox{.}\mkern2mu\raise7\p@\hbox{.}\mkern1mu}}
\makeatother

\titleformat{\section}
{\normalfont\Large\bfseries}{\thesection}{1em}{}
\titleformat*{\subsection}{\Large\bfseries}
\titleformat*{\subsubsection}{\large\bfseries}
\titleformat*{\paragraph}{\large\bfseries}
\titleformat*{\subparagraph}{\large\bfseries}

\makeatother
\theoremstyle{Theorem}
\newtheorem{thm}{Theorem}[section]
\newtheorem{lem}[thm]{Lemma}

\newtheorem{cor}[thm]{Corollary}

\theoremstyle{definition}
\newtheorem{defn}[thm]{Definition}

\newcommand{\N}{\mathbb{N}}
\newcommand{\Z}{\mathbb{Z}}
\newcommand{\F}{\mathcal{P}_{f}\left(\mathbb{N}\right)}
\newcommand{\G}{\mathcal{P}_{f}}

\date{\vspace{-5ex}}

\begin{document}

\title{{\bf Multidimensional Stronger Central Sets Theorem and its Polynomial Extension}}

\author{ 
Sayan Goswami\\  \textit{ sayan92m@gmail.com}\footnote{Ramakrishna Mission Vivekananda Educational and Research Institute, Belur Math,
Howrah, West Benagal-711202, India.}\\
\and
Sourav Kanti Patra\\ \textit{souravkantipatra@gmail.com}\footnote{ Kishori Sinha Mahila College, Aurangabad, Bihar-824101, India.}
}

%\date{\vspace{-5ex}}

\makeatother

\providecommand{\corollaryname}{Corollary}
\providecommand{\definitionname}{Definition}
\providecommand{\examplename}{Example}
\providecommand{\factname}{Fact}
\providecommand{\lemmaname}{Lemma}
\providecommand{\propositionname}{Proposition}
\providecommand{\remarkname}{Remark}
\providecommand{\theoremname}{Theorem}
\newcommand{\RomanNumeralCaps}[1]{\MakeUppercase{\romannumeral #1}}

\maketitle
\begin{abstract}

We establish and fully characterize the multidimensional extension of the Stronger Central Sets Theorem, originally proved in \cite{dhs}. Additionally, we develop a polynomial generalization of this result. Our approach utilizes tools from the algebra of the Stone–\v{C}ech compactification of discrete semigroups. Several applications of these results are also discussed.

\end{abstract}

\section{Introduction}
Throughout this article, we denote by $\mathbb{N}$ the set of all positive integers. 
For any $r \in \mathbb{N}$, an $r$-coloring of a set $X$ is a partition of $X$ into $r$ disjoint subsets. 
A coloring is said to be finite if it corresponds to an $r$-coloring for some $r \in \mathbb{N}$. 
A subset $F \subseteq X$ is called \textbf{monochromatic} if it is entirely contained within a single color class.  

Let $(S, \cdot)$ be a semigroup, and let $\mathcal{F}$ be a family of subsets of $S$. 
The family $\mathcal{F}$ is said to be \textbf{partition regular (P.R.)} if, for any finite coloring of $S$, there exists a member $F \in \mathcal{F}$ that is monochromatic.  

A fundamental result in arithmetic Ramsey theory is \textbf{van der Waerden's theorem} \cite{key-8}, 
which states that the family of all arithmetic progressions of any given length is P.R. over $\mathbb{N}$. 
Another landmark result, due to \textbf{Schur} \cite{sc}, asserts that the family 
\[
\left\lbrace \{x, y, x+y\} : x \neq y \in \mathbb{N}\right\rbrace
\]
is P.R. over $\mathbb{N}$.  

A nonlinear extension of van der Waerden's theorem, known as the \textbf{Polynomial van der Waerden Theorem}, 
was established by \textbf{Bergelson and Leibman} in \cite{bl}. 
They employed techniques from topological dynamics and PET induction to prove this result.\footnote{The authors actually proved a generalized form of the much stronger \textbf{Szemerédi's Theorem}, but this aspect will not be discussed in the present paper.}  

Throughout this article, we denote by $\mathbb{P}$ the set of all polynomials with no constant term.  

\begin{thm}[Polynomial van der Waerden Theorem]\label{pvdw}
Let $r \in \mathbb{N}$, and suppose that $\mathbb{N}$ is partitioned into $r$ color classes, i.e.,  
\[
\mathbb{N} = \bigcup_{i=1}^{r} C_i.
\]  
Then, for any finite collection of polynomials $F \subset \mathbb{P}$, there exist $a, d \in \mathbb{N}$ and some $1 \leq j \leq r$ such that  
\[
\{a + p(d) : p \in F\} \subset C_j.
\]
\end{thm}

 An infinitary extension of \textbf{Schur's theorem} motivates the following definition of \textbf{IP sets}.  

\begin{defn}[\textbf{IP Set} {\cite{key-11}}]
\begin{enumerate}
    \item For any set $X$, let $\mathcal{P}_{f}(X)$ denote the collection of all nonempty finite subsets of $X$.  
    \item Let $(S,+)$ be a commutative semigroup. A subset $A \subseteq S$ is called an \textbf{IP set} if there exists an injective sequence $\langle x_n \rangle_n$ in $S$ such that  
    \[
    A = FS\left(\langle x_n \rangle_n \right) = \left\{ \sum_{n \in \alpha} x_n : \alpha \in \mathcal{P}_{f}(\mathbb{N}) \right\}.
    \]
    For brevity, given $\alpha \in \mathcal{P}_{f}(\mathbb{N})$, we write  
    \[
    x_\alpha = \sum_{n\in \alpha} x_n.
    \]  
\end{enumerate}
\end{defn}

A fundamental result in \textbf{Ramsey theory}, known as the \textbf{Finite Sums Theorem} or \textbf{Hindman’s Theorem} \cite{key-6}, asserts the following:  

\begin{thm}[\textbf{Hindman’s Theorem}]
If $(S,+)$ is a commutative semigroup, then for every finite coloring of $S$, there exists a monochromatic IP set.  
\end{thm}

For a comprehensive discussion on the classical approach to Ramsey theory, we refer the reader to \cite{gxy}. 
In this article, we adopt an approach based on \textbf{ultrafilter theory}.  

Before proceeding, we recall some fundamental results from the \textbf{Stone–Čech compactification} of discrete semigroups, 
a framework often referred to as \textbf{topological algebra} or the \textbf{theory of ultrafilters}.

\subsection{A Brief Review of the Algebra of the Stone-\v{C}ech Compactifications of discrete semigroups}

 Ultrafilters are set-theoretic objects that are intimately related to Ramsey's theory. In this section, we give a brief introduction to this theory.  For details, we refer to the book of N. Hindman and D. Strauss \cite{key-11} to the readers. 
A filter $\mathcal{F}$ over any nonempty set $X$ is a collection of subsets of $X$ such that

\begin{enumerate}
    \item $\emptyset \notin \mathcal{F}$, and $X\in \mathcal{F}$,
    \item $A\in \mathcal{F}$, and $A\subseteq B$ implies $B\in \mathcal{F},$
    \item $A,B\in \mathcal{F}$ implies $A\cap B\in \mathcal{F}.$
\end{enumerate}
Using Zorn's lemma we can guarantee the existence of maximal filters which are called ultrafilters. Any ultrafilter $p$ has the following property:
\begin{itemize}
    \item if $X=\bigcup_{i=1}^rA_i$ is any finite partition of $X$, then there exists $i\in \{1,2,\ldots ,r\}$ such that $A_i\in p.$
\end{itemize}

Let  $(S,\cdot)$ be any discrete semigroup. Let $\beta S$ be the collection of all ultrafilters. For every $A\subseteq S,$ define $\overline{A}=\{p\in \beta S: A\in p\}.$ Now one can check that the collection $\{\overline{A}: A\subseteq S\}$ forms a basis for a topology. This basis generates a topology over $\beta S.$  We can extend the operation $``\cdot "$ of $S$ over $\beta S$  as: for any $p,q\in \beta S,$ $A\in p\cdot q$ if and only if $\{x:x^{-1}A\in q\}\in p.$ With this operation $``\cdot "$, $(\beta S,\cdot)$ becomes a compact Hausdorff right topological semigroup. One can show that $\beta S$ is nothing but the Stone-\v{C}ech compactification of $S.$ Hence Ellis's theorem guarantees that there exist idempotents in $(\beta S,\cdot)$. The set of all idempotents in $(\beta S,\cdot)$ is denoted by $E\left((\beta S,\cdot)\right).$
It can be shown that every member of the idempotents of $(\beta S,\cdot)$ contains an $IP$ set, which means every idempotent witnesses  Hindman's theorem. Using Zorn's lemma one can show that $(\beta S,\cdot)$ contains minimal left ideals (minimal w.r.t. the inclusion). A well-known fact is that the union of such minimal left ideals is a minimal two-sided ideal, denoted by $K(\beta S,\cdot).$ Here we recall a few well-known classes of sets that are relevant to our work. 
%For details readers can see \cite{key-11}.

%Before proceeding, we recall also a few well-known classes of sets that are relevant for our aims.

\begin{defn}\label{rev2defn} Let $(S,\cdot)$ be a semigroup, let $n\in\N$ and let $A\subseteq S$. We say that
\begin{itemize} 
\item $A$ is a {\it thick set} if for any finite subset $F\subset S$, there exists an element $x\in S$ such that $Fx=\{fx:f\in F\}\subset A$;
\item $A$ is a {\it syndetic set} if there exists a finite set $F\subset S$ such that $S=\bigcup_{x\in F}x^{-1}A$, where $x^{-1}A=\{y:xy\in A\}$;
\item $A$ is {\it piecewise syndetic set} if there exists a finite set $F\subset S$ such that $\bigcup_{x\in F}x^{-1}A$ is a thick set. It is well known that $A$  is piecewise syndetic if and only if there exists $p\in K(\beta S,\cdot)$ such that $A\in p.$

\item $A$ is {\it central set} if it belongs to a minimal idempotent in $\beta S$.
\end{itemize} \end{defn}
It can be proved that a set $A$ is thick if and only if there exists a left ideal $L$ such that $L\subseteq \overline{A}$. And a set $A$ is syndetic if and only if for every left ideal $L$, $L\cap \overline{A}\neq \emptyset.$
Note that if $f:(S,\cdot )\rightarrow (T,\cdot)$ be any map between discrete semigroups. Then $\tilde{f}:\beta S\rightarrow \beta T$ be the continuous extension of $f$ defined by $\tilde{f}(p)=\tilde{f}(\lim_{x\rightarrow p} x).$ 
%For details about the theory of ultrafilters, we refer to the book \cite{key-11}.

\subsubsection{Tensor Product of Ultrafilters}
The tensor product of ultrafilters was introduced in \cite{an}.
 In this section, we recall materials from \cite{tensor}.  Let $(S,\cdot )$ and $(T,\cdot )$ be two discrete semigroups.

\begin{defn}
    Let $p\in (\beta S,\cdot )$ and  $q\in (\beta T,\cdot )$ be two ultrafilters. Then the tensor product of $p$ and $q$ is defined as
$$p\otimes q=\{A\subseteq S\times T: \{x\in S: \{y:(x,y)\in A\}\in q\}\in p\}$$
where $x\in S$, $y\in T.$
\end{defn}
Clearly $p\otimes q$ is a member of $\beta (S\times T).$
Tensor products can be characterized in terms of limits as follows.
$$p\otimes q=\lim_{s\rightarrow p}\lim_{t\rightarrow q}(s,t).$$
Using the induction argument, one can extend the above definition of tensor products for multiple ultrafilters. 
%The following definition uses inductive arguments.

\begin{defn}\cite[Definition 1.15]{bhw}
    Let $k\in \N$ and for $i\in \{1,2,\ldots ,k\}$, let $S_i$ be a semigroup and let $p_i\in \beta S_i.$ We define $\bigotimes_{i=1}^kp_i\in \beta \left( \bigtimes_{i=1}^kS_i\right)$ as follows
    \begin{enumerate}
        \item $\bigotimes_{i=1}^1 p_i=p_1,$
        \item given $k\in \N$, and $A\subseteq \bigtimes_{i=1}^{k+1}S_i$, $A\in \bigotimes_{i=1}^{k+1}p_i$ if and only if $$\left\lbrace (x_1\ldots ,x_k)\in \bigtimes_{i=1}^k S_i:\left\lbrace x_{k+1}\in S_{k+1}:(x_1\ldots ,x_{k+1})\in A\right\rbrace\in p_{k+1} \right\rbrace \in \bigotimes_{i=1}^{k}p_i.$$
    \end{enumerate}
\end{defn}
Similarly, we can readdress the above definition in the language of limit:
$$\bigotimes_{i=1}^{k}p_i=\lim_{s_1\rightarrow p_1}\cdots \lim_{s_k\rightarrow p}(s_1,\ldots ,s_k).$$

In \cite{bhw}, V. Bergelson, N. Hindman and K. Williams used the tensor product tricks  to prove multi-dimensional Ramsey theoretic results including Milliken-Taylor theorem, and it's polynomial extensions. For details on multi-dimensional Ramsey theoretic results we refer \cite{ bhw, key-11,lupini}.  In \cite{b}, Beigelb\"{o}ck proved the simultaneous extension of the Central Sets Theorem and Milliken-Taylor theorem using a little different approach described in \cite[Chapter 18]{bhw}. One of our main motivations is to study the tensor product of minimal idempotent ultrafilters, which was not extensively covered earlier. As a consequence, we are able to prove the polynomial extensions of several multidimensional results in Ramsey theory. Before we proceed, we need the following technical Lemma from \cite{bhw} (a variation can be found in \cite{w}), which is essential for our purpose.

Let $X$ be any nonempty set and let $\mathcal{P}(X)$ be the power set of $X.$ 
\begin{lem}[Lifting lemma]\label{Lifting lemma} \cite[Lemma 2.5]{bhw}
    Let $(S,\cdot)$ be a semigroup, let $m\in \N$, and let $p_1,\ldots ,p_m\in \beta S.$ Let $A\in \bigotimes_{j=1}^m p_j.$ Then for $j\in \{1,2,\ldots ,m\}$ there exists $D_j:S^{j-1}\rightarrow \mathcal{P}(S)$ such that 
    \begin{enumerate}
        \item for $j\in \{1,2,\ldots ,m\},$  if for each $s\in \{1,2,\ldots ,j-1\},$  $w_s\in D_s(w_1,\ldots ,w_{s-1}),$  then $D_j(w_1,\ldots ,w_{j-1})\in p_j;$

        \item if for each $s\in \{1,2,\ldots ,m\},$  $w_s\in D_s(w_1,\ldots ,w_{s-1}),$ then $(w_1,\ldots ,w_m)\in A.$
    \end{enumerate}
\end{lem}

\subsection{The Central Sets Theorem and its Polynomial Extensions}

Before we proceed to our main results, we need to recall the Central Sets Theorem. After the foundation of both van der Waerden's and Hindman's theorem, an immediate question appeared if one can find a joint extension of both of these theorems. In \cite{f2}, using the methods of Topological dynamics, Furstenberg defined the notions of Central Sets and proved that if $\N$ is finitely colored, then one of the color classes is Central. Then he proved a joint extension of both the van der Waerden's and Hindman's theorem, known as the Central Sets Theorem.

\begin{thm} [Central Sets Theorem]\label{cst} \cite[Proposition 8.21]{f2}
Let $l\in\mathbb{N}$, and $A\subseteq\mathbb{N}$ be a central set. For
each $i\in\{1,2,\ldots,l\}$ let $\langle x_{i,m}\rangle_{m=1}^{\infty}$ be a sequence in $\mathbb{\mathbb{Z}}$. Then there exist  sequences
$\langle b_{m}\rangle_{m=1}^{\infty}$ in $\mathbb{N}$ and $\langle K_{m}\rangle_{m=1}^{\infty}$
in $\F$ such that
\begin{enumerate}
\item For each $m$, $\max K_{m}<\min K_{m+1}$ and
\item For each $i\in\{1,2,\ldots,l\}$ and $H\in\F$,
$\sum_{m\in H}(b_{m}+\sum_{t\in K_{m}}x_{i,t})\in A$.
\end{enumerate}
    
\end{thm}

%Later in \cite{bh1,new2} authors established the relation between the Central sets and the minimal idempotents mentioned in the Definition \ref{rev2defn}.
%Then using the properties of the minimal idempotents, the Central Sets Theorem was generalized in several directions. For details about the developments and their combinatorial applications,  we refer to the article \cite{new1}.  
Before we state the stronger version of Theorem \ref{cst}, we need the following definitions from \cite{dhs}.

\begin{defn} ($J$-set)
    \begin{enumerate}
        \item  $A$ is called a
$J$-set if and only if  for every $H\in \G(^\N\N)$, there exists $a\in \N$ and $\beta\in \F$ such that for all $f\in H,$ $a+\sum_{t\in \beta}f(t)\in A.$

        \item $\mathcal{J}=\{p\in \beta \N :(\forall A\in p)\, A \text{ is a } J \text{ set}\}.$
               
    \end{enumerate}
\end{defn}

It can be proved that $J$ sets are partition regular, and so $\mathcal{J}\neq \emptyset$. Again it is easy to verify that $(\mathcal{J},+)$ is a two sided ideal of $(\beta \N,+)$, hence contains $K(\beta \N,+).$
\begin{itemize}
    \item A set will be called $C$-set if it is a member of any idempotent in $(\mathcal{J},+)$.
    \end{itemize}
As every piecewise syndetic set is a $J$ set, one can show that every Central set is a $C$ set. The following stronger version of the Central Sets Theorem was due to De, Hindman, and Strauss \cite{dhs}.

\begin{thm} [Stonger Central Sets Theorem, \cite{dhs}]\label{scst} Let $\tau=^\N\mathbb{N}$, and let $A \subseteq \N$. 

\begin{enumerate}
    \item $A$ is a $C$ set.

\item there exists functions $\alpha : \mathcal{P}_f(^\N\mathbb{N})\to \mathbb{N},$ and $H: \mathcal{P}_f(^\N\mathbb{N}) \to \mathcal{P}_f\left(\mathbb{N}\right)$ such that 
\begin{enumerate}
\item \label{1.41} if $F,G \in \mathcal{P}_f(\tau)$ and $F \subsetneqq G$ then $\max H(F) < \min H(G)$, and 
\item \label{1.42} whenever $m \in \mathbb{N}$, $G_1, G_2, \ldots , G_m \in \mathcal{P}_f(\tau)$, $G_1 \subsetneq G_2 \subsetneq \cdots \subsetneq G_m$ and for each $i \in \{1,2, \ldots , m\}$, $f_i \in G_i$, one has $$\sum_{i=1}^{m}\bigl(\alpha(G_i)+\sum_{t\in H(G_i)}f_i(t)\bigr)\in A.$$
\end{enumerate}
\end{enumerate}
\end{thm}
%In fact it was also shown in \cite{dhs} that if a set $A\subset \N$ contains the conclusion of Theorem \ref{scst}, then it is a $C$ set.

\subsubsection{Polynomial Extension of the Central Sets Theorem}
In \cite{key-50, key-51}, Hindman and McCutcheon extensively studied the algebraic structure of VIP systems which are the generalizations of usual polynomials. They proved several results including the polynomial extension of Theorem \ref{cst}. This was the joint extension of the polynomial van der Waerden's theorem and the central sets theorem.  %For a relatively weaker version readers can see \cite{bjm}. 
%The polynomial van der Waerden's theorem and the central sets theorem have a joint extension, proven by Hindman and McCutcheon in \cite{key-50}. 

\begin{thm}[Polynomial Central Sets Theorem]\cite{key-51}
 \label{PCST} 
 Let $F\in \G(^\N\N)$, let $T\in \G(\mathbb{P})$ and let $A$ be a central subset of $\N$. Then there exist sequences $\langle b_n\rangle_{n=1}^\infty$ in $\N$ and $\langle H_n\rangle_{n=1}^\infty$ in $\F$ such that
 \begin{enumerate}
     \item for each $n\in \N$, $\max H_n<\min H_{n+1}$ and
     \item for each $f\in F,$ each $P\in T$ and each $K\in \F$ 
 \end{enumerate}
 $$\sum_{n\in K}b_n+P\left(\sum_{n\in K} \sum_{t\in H_n} f(t)\right)\in A.$$

\end{thm}

Later in \cite{ejc}, the authors found the polynomial extension of Theorem \ref{scst}, which is the stronger version of Theorem \ref{PCST}. To state this result we need the following definition which polynomialize the notion of $J$ sets.

\begin{defn}  ($J_{p}$-set:)
    \begin{enumerate}
        \item $A$ is called a
$J_{p}$-set if and only if  for every $F\in\G(\mathbb{P})$, and every $H\in \G(^\N\N)$, there exists $a\in \N$ and $\beta\in \F$ such that for all $P\in F$ and all $f\in H,$ $a+P\left(\sum_{t\in \beta}f(t)\right)\in A$

  \item $\mathcal{J}_p=\{p\in \beta \N :(\forall A\in p)\, A \text{ is a } J_p \text{ set}\}.$
    \end{enumerate}
\end{defn}

We don't know whether $J_p$ sets are P.R., but it can be shown that every piecewise syndetic set is a $J_p$ set. Hence $\mathcal{J}_p\neq \emptyset$. Now one can show that $(\mathcal{J}_p,+)$ is a two sided ideal of $(\beta \N,+)$. We call a set a $C_p$ set if it belongs to some idempotents of $(\mathcal{J}_p,+)$.

\begin{thm}  [Stronger polynomial Central Sets Theorem]\label{no1}\cite[Theorem 11]{ejc} 
Let $A$ be a $C_p$-set and let $T\in \mathcal{P}_f(\mathbb{P})$. There exist functions
$\alpha : \mathcal{P}_f(^{\mathbb{N}}\mathbb{N})\to \mathbb{N}$ and $H: \mathcal{P}_f(^{\mathbb{N}}\mathbb{N}) \to \mathcal{P}_f(\mathbb{N})$ such that
\begin{enumerate}
\item  if $G,K \in \mathcal{P}_f(^{\mathbb{N}}\mathbb{N})$ and $G \subsetneqq K$ then $\max H(G) < \min H(K)$ and
\item  if $n\in \mathbb{N}$, $G_1, G_2, \cdots, G_n \in \mathcal{P}_f(^{\mathbb{N}}\mathbb{N})$,
$G_1 \subsetneq G_2 \subsetneq \cdots \subsetneq G_n$ and for all  $i \in \{1, 2, \ldots, n\}$,
$f_i \in G_i$, then for all $P\in T$,
$$\sum_{i=1}^{n}\alpha(G_i)+P\left(\sum_{i=1}^{n}\sum_{t\in H(G_i)}f_i(t)\right)\in A.$$
\end{enumerate}
\end{thm}

However likewise in Theorem \ref{scst}, the converse direction of Theorem \ref{no1} is not known.
The proof of Theorem \ref{no1} uses induction argument. By observing that any $\beta\in\F$ is a subset of $\{1,\dots,n\}$ for some large enough $n\in\mathbb{N}$, by going along the lines of the proof of Theorem \ref{no1}, one can deduce the following seemingly stronger version of Theorem~\ref{no1}.

\begin{cor}\label{yad}\cite[Corollary 12]{ejc}
Let $A$ be a $C_{p}$-set and $T\in \G(\mathbb{P})$. Then there exist $\alpha:\G(^\N\N)\rightarrow \mathbb{N}, H:\G(^\N\N)\rightarrow \G(\mathbb{N})$ such that

\begin{enumerate}
\item if $F,G\in \G(^\N\N),F\subset G$, then $\max H(F)<\min H(G)$;

\item if $\langle G_{n} \rangle_{n\in \mathbb{N}}$ is a sequence in $\G(^\N\N)$ such that $G_{1}\subsetneq G_{2}\subsetneq\cdots\subsetneq G_{n} \subsetneq \cdots$ and
$f_{i}\in G_{i},i\in \mathbb{N}$, then for all $P\in T,$

\[
\sum_{i\in \beta } \alpha(G_{i})+P\left(\sum_{i\in \beta}\sum_{t\in H(G_{i})}f_{i}(t)\right)\in A.
\]
for all $\beta \in \F$.
\end{enumerate}
\end{cor}

\subsection{Milliken-Taylor Theorem and its Twofold Extensions}

In $1975$-$76$, Milliken and Taylor independently proved the joint extension of both Hindman theorem and Ramsey's theorem. To state the Milliken-Taylor Theorem, we need to introduce some notation. Given $F,G\in \G(\N)$, we write $F<G$ to mean that $\max F< \min G.$ When we say that a sequence $\langle H_n\rangle_n$ in $\G(\N)$ is increasing if for every $n\in \N$, $\max H_n<\min H_{n+1}.$ We present the different terminology because these special cases arise frequently.

\begin{defn} \cite[Definition 1.3]{bhw}
    Let $k\in \N$. 
    \begin{enumerate}
        \item  For any set $X,$ $[X]^k=\{A\subseteq X:|A|=k\}.$
        \item  For a sequence $\langle x_n\rangle_n$ in $\N$ $$[FS(\langle x_n\rangle_n)]_{<}^k=\left\lbrace \left(\sum_{t\in F_1}x_t,\sum_{t\in F_2}x_t,\ldots ,\sum_{t\in F_k}x_t\right):F_1<F_2<\cdots <F_k\right\rbrace.$$
        \item In a semigroup $(S,\cdot ),$ $FP(\langle y_n\rangle_n)$ is a {\it product subsystem} of $FP(\langle x_n\rangle_n)$  if and only if there exists an increasing sequence  $\langle H_n\rangle_n$ in $\G(\N)$   such that for each $n\in \N,$ $y_n=\Pi_{t\in H_n}x_t$ where the products $\Pi_{t\in H_n}x_t$ are
        computed in increasing order of indices. If the semigroup is commutative, we denote it by $(S,+)$ and the product subsystem $FP(\langle y_n\rangle_n)$ is denoted by $FS(\langle y_n\rangle_n)$ referred by sum subsystem.
    \end{enumerate}
\end{defn}

One of the fundamental theorems of Ramsey theory was due to Ramsey \cite{ramsey}: 
\begin{thm}[Ramsey]
If $X$ is an infinite set, and $k\in \N$, then for every finite coloring of $[X]^k$, there exists an infinite set $Y\subseteq X$ such that the set $[Y]^k$ is monochromatic.
\end{thm}
The following theorem was due to K. Milliken \cite{m} and A. Taylor \cite{t} generalizes both Ramsey's theorem and Hindman's theorem, known as the Milliken-Taylor theorem. They originally proved two different but equivallent versions of the Milliken-Taylor theorem (see \cite{bhw}). Here we quote the version from \cite{m}.

\begin{thm}[Milliken-Taylor]\label{mt}
    Let $m,r\in \N$, and $[\N]^m=\bigcup_{i=1}^rC_i$ be a $r$-coloring. Let $\langle x_n\rangle_n$ be a sequence in $\N$. Then there exists $i\in \{1,2,\ldots ,r\}$, and a sum subsystem  $FS(\langle y_n\rangle_n)$ of $FS(\langle x_n\rangle_n)$ such that  $[FS(\langle y_n\rangle_n)]_{<}^m\subseteq C_i.$ 
\end{thm}

To explain several extensions of the Milliken-Taylor theorem, we need the following notions of {\it Milliken-Taylor system.} 

\begin{defn}[Milliken-Taylor system]\label{defmillikentaylor}
    Let $m\in \N,$ let $\langle a_j\rangle_{j=1}^m$ and $\langle x_n\rangle_{n}$ be sequences in $\N.$ The
Milliken-Taylor System determined by $\langle a_j\rangle_{j=1}^m$ and $\langle x_n\rangle_{n}$ is 
$$MT\left(\langle a_j \rangle_{j=1}^m, \langle x_n\rangle_{n}\right)=\left\lbrace \sum_{j=1}^m a_j \cdot \sum_{t\in F_j}x_t:F_1<F_2<\cdots <F_m \right\rbrace.$$
\end{defn}

The following theorem is another version of Theorem \ref{mt} that uses linear combination of idempotent ultrafilters. 

\begin{thm}\cite[Theorem 1.11.]{bhw}\label{mt2}
    Let $K\in \N,$ let  $\langle a_j\rangle_{j=1}^k,$ and $\langle x_n\rangle_{n}$ be sequences in $\N.$ Let $g(z)=\sum_{j=1}^k a_jz$, and let $A\subseteq \N$. The following statements are equivalent.

    \begin{enumerate}
        \item There is an idempotent $p\in \bigcap_{m=1}^\infty \overline{FS\left(\langle x_n\rangle_{n=m}^\infty\right)}$ such that $A\in \tilde{g}(p)$

        \item There is a sum subsystem $FS\left(\langle y_n\rangle_{n=1}^\infty\right)$ of  $FS\left(\langle x_n\rangle_{n=1}^\infty\right)$ such that $MT\left(\langle a_j\rangle_{j=1}^m, \langle x_n\rangle_{n}\right)\subseteq A.$
    \end{enumerate}

\end{thm}

Using the tensor product of ultrafilters one can find a seemingly stronger version of the Milliken-Taylor theorem.

\begin{thm}\label{mt1}\textup{\cite[Theorem 1.17, Milliken-Taylor Theorem \MakeUppercase{\romannumeral 1}]{bhw}}
    Let $(S,+)$ be a commutative semigroup, and $m\in \N,$ and let $A\subseteq \bigtimes_{i=1}^m S.$  Then the
following statements are equivalent.
\begin{enumerate}
    \item There is a sequence $\langle x_n\rangle_{n=1}^\infty$ in $S$ such that
    $$\left\lbrace \left(\sum_{t\in F_1}x_t,\sum_{t\in F_2}x_t,\ldots ,\sum_{t\in F_k}x_t\right):F_1<F_2<\cdots <F_k \right\rbrace \subseteq A.$$
    \item There is an idempotent $p\in \beta S$ such that $A \in \bigotimes_{i=1}^m p.$
\end{enumerate}
\end{thm}

The following theorem is another version of the Milliken-Taylor theorem. Here one considers tensor products of different idempotent ultrafilters.
\begin{thm}\textup{\cite[Theorem 1.16, Milliken-Taylor Theorem \MakeUppercase{\romannumeral 2}]{bhw}}\label{newaddmt}
    Let $(S,+)$ be a commutative semigroup, and $m\in \N,$ and let $A\subseteq \bigtimes_{i=1}^m S.$  The
following statements are equivalent.
\begin{enumerate}
    \item For each $i\in \{1,2,\ldots ,m\}$,  there exist sequences $\langle x_{i,n}\rangle_{n=1}^\infty$ in $S$ such that
    $$\left\lbrace \left(\sum_{t\in F_1}x_{1,t},\sum_{t\in F_2}x_{2,t},\ldots ,\sum_{t\in F_k}x_{m,t}\right):F_1<F_2<\cdots <F_k \right\rbrace \subseteq A.$$
    \item There exist idempotents $p_1,\ldots ,p_m\in \beta S$ such that $A \in \bigotimes_{i=1}^m p_i.$
\end{enumerate}
\end{thm}

In \cite{b}, Beigelb\"{o}ck simultaneously extends the combinatorial consequences of both of the Theorem \ref{cst} and Theorem \ref{mt1}. Unfortunately, Beigelb\"{o}ck's result does not address the ultrafilter characterization of sets witnessing his result. Let us first state his result, and then we address our results.

\begin{thm}\label{bei}
    Let $(S,+)$ be a commutative semigroup and assume that there exists a nonprincipal minimal idempotent in $\beta S.$  For each $l\in \N$, let $\langle y_{l,n}\rangle_{n}$ be a sequence in $S.$ Let $k,r\in \N$ and let $[S]^k=\bigcup_{i}^rA_i.$ There exist $i\in \{1,2,\ldots ,r\}$ a sequence $\langle a_n\rangle_n$ in $S$ and a sequence $\alpha_0<\alpha_1<\cdots <\alpha_n<\cdots$ in $\G(\N)$ such that $\Big[ FS\big(\langle a_n\prod_{t\in \alpha_n}y_{g(n),t}\rangle_n\big)\Big]_<^k\subset A_i$
\end{thm}

Our first result is the improvement of Theorem \ref{bei}. We simultaneously extend Theorem \ref{scst} and Theorem \ref{mt1} both in the algebraic and combinatorial direction.  In our article, {\it we restrict ourselves to the set of natural numbers, but this technique can be adapted for any discrete commutative semigroups.}

\begin{thm}[Multidimensional Stronger Central Sets Theorem \MakeUppercase{\romannumeral 1}]\label{todo0}
    Let $m\in \N$, and $A\subseteq \N^m$. Then the following two statements are equivalent.

  \begin{enumerate}
      \item There exists functions $\alpha : \mathcal{P}_f(^\N\mathbb{N})\to \mathbb{N}$ and $H: \mathcal{P}_f(^\N\mathbb{N}) \to \mathcal{P}_f\left(\mathbb{N}\right)$ such that 
         \begin{enumerate}
             \item \label{1.41} if $F,G \in \mathcal{P}_f(\tau)$ and $F \subsetneq G$ then $\max H(F) < \min H(G)$, and 
            \item \label{1.42} if $n (\ge m)\in\mathbb N$, $G_1\subsetneq G_2\subsetneq\dots\subsetneq G_n$, then for each $\beta_1<\beta_2<\dots<\beta_m\le\{n\},$ 
            in $\G(\N)$ we have 
            $$\left\lbrace \left( \sum_{i\in \beta_1}\left(\alpha(G_i)+\sum_{t\in H(G_i)}f_i(t)\right),\cdots ,\sum_{i\in \beta_m}\left(\alpha(G_i)+\sum_{t\in H(G_i)}f_i(t)\right):f_i\in G_i \text{ for every }i\in \N \right)\right\rbrace\subset A.$$
        \end{enumerate}
     \item There exists an  idempotent $p$ in $(\mathcal{J},+)$, such that $A\in \bigotimes_{i=1}^m p.$

\end{enumerate}  
    
\end{thm}
Similar to the Corollary \ref{yad}, the following corollary of the above theorem gives an infinitary version.

\begin{cor}\label{todo1}
    Let $m\in \N$, and $A\subseteq \N^m$. Then the following two statements are equivalent.

  \begin{enumerate}
      \item there exists functions $\alpha : \mathcal{P}_f(^\N\mathbb{N})\to \mathbb{N}$ and $H: \mathcal{P}_f(^\N\mathbb{N}) \to \mathcal{P}_f\left(\mathbb{N}\right)$ such that 
         \begin{enumerate}
             \item \label{1.41} if $F,G \in \mathcal{P}_f(\tau)$ and $F \subsetneq G$ then $\max H(F) < \min H(G)$, and 
            \item \label{1.42} whenever $(G_i)_{i} \in \mathcal{P}_f(\tau)$, is a sequence of functions, i.e. $G_1 \subsetneq G_2 \subsetneq \cdots \subsetneq G_n\subsetneq \cdots$, then for every $\beta_1<\cdots<\beta_m$ in $\G(\N)$ we have 
            $$\left\lbrace \left( \sum_{i\in \beta_1}\left(\alpha(G_i)+\sum_{t\in H(G_i)}f_i(t)\right),\cdots ,\sum_{i\in \beta_m}\left(\alpha(G_i)+\sum_{t\in H(G_i)}f_i(t)\right):f_i\in G_i \text{ for every }i\in \N \right)\right\rbrace\subset A.$$
        \end{enumerate}
     \item There exists an  idempotent $p$ in $(\mathcal{J},+)$, such that $A\in \bigotimes_{i=1}^m p.$

\end{enumerate}  
    
\end{cor}

Similarly to Theorem \ref{newaddmt}, the following theorem is a simultaneous extension of Theorem \ref{scst}, and Theorem \ref{newaddmt}.

\begin{thm}[Multidimensional Stronger Central Sets Theorem \MakeUppercase{\romannumeral 2}]\label{newmul1}
    Let $m\in \N$, and $A\subseteq \N^m$. Then the following two statements are equivalent.

  \begin{enumerate}
      \item There exists functions $\alpha_1,\ldots ,\alpha_m : \mathcal{P}_f(^\N\mathbb{N})\to \mathbb{N}$ and $H_1,\ldots ,H_m: \mathcal{P}_f(^\N\mathbb{N}) \to \mathcal{P}_f\left(\mathbb{N}\right)$ such that 
         \begin{enumerate}
             \item \label{1.41} if $F,G \in \mathcal{P}_f(\tau)$ and $F \subsetneq G$ then for every $i\in \{1,2,\ldots ,m\}$, $\max H_i(F) < \min H_i(G)$, and 
            \item \label{1.42} if $n (\ge m)\in\mathbb N$, $G_1\subsetneq G_2\subsetneq\dots\subsetneq G_n$, then for each $\beta_1<\beta_2<\dots<\beta_m\le\{n\},$ 
            in $\G(\N)$ we have 
            $$\left\lbrace \left( \sum_{i\in \beta_1}\left(\alpha_1(G_i)+\sum_{t\in H_1(G_i)}f_i(t)\right),\cdots ,\sum_{i\in \beta_m}\left(\alpha_m(G_i)+\sum_{t\in H_m(G_i)}f_i(t)\right):f_i\in G_i \text{ for every }i\in \N \right)\right\rbrace\subset A.$$
        \end{enumerate}
     \item There exist idempotents $p_1,\ldots ,p_m$ in $(\mathcal{J},+)$, such that $A\in \bigotimes_{i=1}^m p_i.$

\end{enumerate}  
    
\end{thm}
The main difference between Theorem \ref{todo0}, and Theorem \ref{newmul1} is that in Theorem \ref{todo0}, our functions $\alpha,$ and $H$ are same throughout all the coordinates, but in Theorem \ref{newmul1},  our functions $\alpha,$ and $H$ are different throughout all the coordinates. The main reason is that we used different idempotent ultrafilters.
In a similar fashion to Corollary \ref{todo1}, the following corollary addresses the infinitary version of Theorem \ref{newmul1}.

\begin{cor}\label{newmul2}
    Let $m\in \N$, and $A\subseteq \N^m$. Then the following two statements are equivalent.

  \begin{enumerate}
      \item There exists functions $\alpha_1,\ldots ,\alpha_m : \mathcal{P}_f(^\N\mathbb{N})\to \mathbb{N}$ and $H_1,H_2\ldots ,H_m: \mathcal{P}_f(^\N\mathbb{N}) \to \mathcal{P}_f\left(\mathbb{N}\right)$ such that 
         \begin{enumerate}
             \item \label{1.41} if $F,G \in \mathcal{P}_f(\tau)$ and $F \subsetneq G$ then $\max H(F) < \min H(G)$, and 
            \item \label{1.42} whenever $(G_i)_{i} \in \mathcal{P}_f(\tau)$, is a sequence of functions, i.e. $G_1 \subsetneq G_2 \subsetneq \cdots \subsetneq G_n\subsetneq \cdots$, then for every $\beta_1<\cdots<\beta_m$ in $\G(\N)$ we have 
            $$\left\lbrace \left( \sum_{i\in \beta_1}\left(\alpha_1(G_i)+\sum_{t\in H_1(G_i)}f_i(t)\right),\cdots ,\sum_{i\in \beta_m}\left(\alpha_m(G_i)+\sum_{t\in H_m(G_i)}f_i(t)\right):f_i\in G_i \text{ for every }i\in \N \right)\right\rbrace\subset A.$$
        \end{enumerate}
     \item There exist idempotents $p_1,\ldots ,p_m$ in $(\mathcal{J},+)$, such that $A\in \bigotimes_{i=1}^m p_i.$

\end{enumerate}  
    
\end{cor}

Our next result polynomialize Theorem \ref{todo0}. Unfortunately, we are unable to find the ultrafilter characterization of the next theorem. The main problem is that we don't know any ultrafilter characterization of the polynomial Central Sets Theorem. 

\begin{thm}[Multidimensional Stronger Polynomial Central Sets Theorem \MakeUppercase{\romannumeral 1}]\label{todo2}
     Let $m\in \N$, and let $p$ be a minimal idempotent in $(\beta\mathbb N,+)$. Then for every $A\in \bigotimes_{i=1}^m p$, and $T\in\mathcal P_f(P)$, we have two functions $\alpha:\mathcal P_f(^\N\N)\rightarrow \mathbb N$, and $H:\mathcal P_f(^\N\N)\rightarrow \mathcal P_f(\mathbb N)$ such that

 \begin{enumerate}
     \item if $G,K\in\mathcal P_f(^\N\N)$, and $G\subsetneq K$, then $H(G)< H(K)$, and

     \item if $n (\ge m)\in\mathbb N$, $G_1\subsetneq G_2\subsetneq\dots\subsetneq G_n$, then for each $\beta_1<\beta_2<\dots<\beta_m\le\{n\},$ in $\G(\N)$, $ f_i\in G_i$, and $P_1,\ldots ,P_m\in T$ we have 
     \[
     \left(\sum_{i\in\beta_1}  \alpha(G_i) + P_1\left(\sum_{i\in\beta_1} \sum_{t\in H(G_i)} f_i(t)\right),\ldots, \sum_{i\in\beta_m} \alpha (G_i)+P_m\left(\sum_{i\in\beta_m}\sum_{t\in H(G_i)}f_i(t)\right)\right) \in A.
     \]
 \end{enumerate}
 In fact we may consider $p\in E\left(\mathcal{J}_p,+\right).$
\end{thm}

The following corollary of the above theorem gives an infinitary version of Theorem \ref{todo2}. This result is the polynomial extension of Corollary \ref{todo1}.
\begin{cor}\label{todo3}
     Let $m\in \N$, and let $p$ be a minimal idempotent in $(\beta\mathbb N,+)$. Then for every $A\in \bigotimes_{i=1}^m p$, and $T\in\mathcal P_f(P)$, we have two functions $\alpha:\mathcal P_f(^\N\N)\rightarrow \mathbb N$, and $H:\mathcal P_f(^\N\N)\rightarrow \mathcal P_f(\mathbb N)$ such that

 \begin{enumerate}
     \item if $G,K\in\mathcal P_f(^\N\N)$, and $G\subsetneq K$, then $H(G)< H(K)$, and

     \item if $G_1\subsetneq G_2\subsetneq\dots\subsetneq G_n\subsetneq\cdots$ in $P_f(^\N\N)$, then for each $\beta_1<\beta_2<\dots<\beta_m,$ in $\G(\N)$,  $ f_i\in G_i$, and $P_1,\ldots ,P_m\in T$ we have 
     \[
     \left(\sum_{i\in\beta_1}  \alpha(G_i) + P_1\left(\sum_{i\in\beta_1} \sum_{t\in H(G_i)} f_i(t)\right),\ldots, \sum_{i\in\beta_m} \alpha (G_i)+P_m\left(\sum_{i\in\beta_m}\sum_{t\in H(G_i)}f_i(t)\right)\right) \in A.
     \]
 \end{enumerate}
 In fact we may consider $p\in E\left(\mathcal{J}_p,+\right).$
\end{cor}

The following theorem is the polynomial extension of Theorem \ref{newmul1}. Here we considered tensor product of different minimal idempotent ultrafilters.

\begin{thm}[Multidimensional Stronger Polynomial Central Sets Theorem  \MakeUppercase{\romannumeral 2}] \label{newmulpoly1}
     Let $m\in \N$, and let $p_1,\ldots ,p_m$ be  minimal idempotents in $(\beta\mathbb N,+)$. Then for every $A\in \bigotimes_{i=1}^m p_i$, and $T\in\mathcal P_f(P)$, we have functions $\alpha_1,\ldots ,\alpha_m:\mathcal P_f(^\N\N)\rightarrow \mathbb N$, and $H_,\ldots ,H_m:\mathcal P_f(^\N\N)\rightarrow \mathcal P_f(\mathbb N)$ such that

 \begin{enumerate}
     \item if $G,K\in\mathcal P_f(^\N\N)$, and $G\subsetneq K$, then for every $i\in \{1,\ldots ,m\}$, we have $H_i(G)< H_i(K)$, and

     \item if $n (\ge m)\in\mathbb N$, $G_1\subsetneq G_2\subsetneq\dots\subsetneq G_n$, then for each $\beta_1<\beta_2<\dots<\beta_m\le\{n\},$ in $\G(\N)$, $ f_i\in G_i$, and $P_1,\ldots ,P_m\in T$ we have 
     \[
     \left(\sum_{i\in\beta_1}  \alpha_1(G_i) + P_1\left(\sum_{i\in\beta_1} \sum_{t\in H_1(G_i)} f_i(t)\right),\ldots, \sum_{i\in\beta_m} \alpha_m (G_i)+P_m\left(\sum_{i\in\beta_m}\sum_{t\in H_m(G_i)}f_i(t)\right)\right) \in A.
     \]
 \end{enumerate}
 In fact we may consider $p_1,\ldots ,p_m\in E\left(\mathcal{J}_p,+\right).$
\end{thm}

The following corollary of the above theorem gives an infinitary version of Theorem \ref{newmulpoly1}. This result is the polynomial extension of Corollary \ref{newmul2}.
\begin{cor}\label{newmulpoly2}
     Let $m\in \N$, and let $p_1,\ldots ,p_m$ be  minimal idempotents in $(\beta\mathbb N,+)$. Then for every $A\in \bigotimes_{i=1}^m p_i$, and $T\in\mathcal P_f(P)$, we have functions $\alpha_1,\ldots ,\alpha_m:\mathcal P_f(^\N\N)\rightarrow \mathbb N$, and $H_1,\ldots ,H_m:\mathcal P_f(^\N\N)\rightarrow \mathcal P_f(\mathbb N)$ such that

 \begin{enumerate}
     \item if $G,K\in\mathcal P_f(^\N\N)$, and $G\subsetneq K$, then for every $i\in \{1,2,\ldots ,m\}$, we have $H_i(G)< H_i(K)$, and

     \item if $G_1\subsetneq G_2\subsetneq\dots\subsetneq G_n\subsetneq\cdots$ in $P_f(^\N\N)$, then for each $\beta_1<\beta_2<\dots<\beta_m,$ in $\G(\N)$,  $ f_i\in G_i$, and $P_1,\ldots ,P_m\in T$ we have 
     \[
     \left(\sum_{i\in\beta_1}  \alpha_1(G_i) + P_1\left(\sum_{i\in\beta_1} \sum_{t\in H_1(G_i)} f_i(t)\right),\ldots, \sum_{i\in\beta_m} \alpha_m (G_i)+P_m\left(\sum_{i\in\beta_m}\sum_{t\in H_m(G_i)}f_i(t)\right)\right) \in A.
     \]
 \end{enumerate}
 In fact we may consider $p_1,\ldots ,p_m\in E\left(\mathcal{J}_p,+\right).$
\end{cor}

\subsubsection{Polynomial Extension of the Milliken-Taylor Theorem}

Let $m\in \N$, and let $h:\N^m\rightarrow \N$ be any polynomial having no constant term. Then from \cite[Theorem 3.2]{bhw}, it is clear that for every $p\in \beta \N$, $\tilde{h}\left(\bigotimes_{i=1}^m p\right)=h(p,\ldots ,p).$  The following  theorem is a corollary of \cite[Theorem 3.2]{bhw}, and is the polynomial extension of Theorem \ref{mt1}.

\begin{thm}\cite[Theorem 1.11.]{bhw}\label{mt2}
    Let  $m\in \N,$ $h:\N^m\rightarrow \N$ be any polynomial having no constant term, and let $A\subseteq \bigtimes_{i=1}^m \N.$  Then the 
following statements are equivalent.
\begin{enumerate}
    \item There is a sequence $\langle x_n\rangle_{n=1}^\infty$ in $\N$ such that
    $$\left\lbrace h\left(\sum_{t\in F_1}x_t,\sum_{t\in F_2}x_t,\ldots ,\sum_{t\in F_k}x_t\right):F_1<F_2<\cdots <F_k \right\rbrace \subseteq A.$$
    \item There is an idempotent $p\in \beta S$ such that $A \in \tilde{h}\left(\bigotimes_{i=1}^m p\right).$
\end{enumerate}

\end{thm}

The following corollary immediately follows from \cite[Theorem 3.2]{bhw}, and Corollary \ref{todo1}, and can be thought of as the polynomial extension of Corollary \ref{todo1}.

\begin{cor}\label{todo4}
    Let  $m\in \N,$ $h:\N^m\rightarrow \N$ be any polynomial having no constant term, and let $A\subseteq \bigtimes_{i=1}^m \N.$  Then the 
following statements are equivalent. 

  \begin{enumerate}
      \item there exists functions $\alpha : \mathcal{P}_f(^\N\mathbb{N})\to \mathbb{N}$ and $H: \mathcal{P}_f(^\N\mathbb{N}) \to \mathcal{P}_f\left(\mathbb{N}\right)$ such that 
         \begin{enumerate}
             \item \label{1.41} if $F,G \in \mathcal{P}_f(\tau)$ and $F \subsetneq G$ then $\max H(F) < \min H(G)$, and 
            \item \label{1.42} whenever $(G_i)_{i} \in \mathcal{P}_f(\tau)$, is a sequence of functions, i.e. $G_1 \subsetneq G_2 \subsetneq \cdots \subsetneq G_n\subsetneq \cdots$, then for every $\beta_1<\cdots<\beta_m$ in $\G(\N)$ we have 
            $$\left\lbrace h\left( \sum_{i\in \beta_1}\left(\alpha(G_i)+\sum_{t\in H(G_i)}f_i(t)\right),\cdots ,\sum_{i\in \beta_m}\left(\alpha(G_i)+\sum_{t\in H(G_i)}f_i(t)\right):f_i\in G_i \text{ for every }i\in \N \right)\right\rbrace\in A.$$
        \end{enumerate}
     \item There exists an  idempotent $p$ in $(\mathcal{J},+)$, such that $A\in \tilde{h}\left(\bigotimes_{i=1}^m p\right).$

\end{enumerate}  
    
\end{cor}

In light of Corollary \ref{newmul2}, the following corollary is another version of Corollary \ref{todo4}.

\begin{cor}\label{todo5}
    Let  $m\in \N,$ $h:\N^m\rightarrow \N$ be any polynomial having no constant term, and let $A\subseteq \bigtimes_{i=1}^m \N.$  Then the 
following statements are equivalent. 

  \begin{enumerate}
      \item there exists functions $\alpha_1,\ldots ,\alpha_m : \mathcal{P}_f(^\N\mathbb{N})\to \mathbb{N}$ and $H_1,\ldots ,H_m: \mathcal{P}_f(^\N\mathbb{N}) \to \mathcal{P}_f\left(\mathbb{N}\right)$ such that 
         \begin{enumerate}
             \item \label{1.41} if $F,G \in \mathcal{P}_f(\tau)$ and $F \subsetneq G$ then for each $i\in \{1,2,\ldots ,m\}$, we have $\max H_i(F) < \min H_i(G)$, and 
            \item \label{1.42} whenever $(G_i)_{i} \in \mathcal{P}_f(\tau)$, is a sequence of functions, i.e. $G_1 \subsetneq G_2 \subsetneq \cdots \subsetneq G_n\subsetneq \cdots$, then for every $\beta_1<\cdots<\beta_m$ in $\G(\N)$ we have 
            $$\left\lbrace h\left( \sum_{i\in \beta_1}\left(\alpha_1(G_i)+\sum_{t\in H_1(G_i)}f_i(t)\right),\cdots ,\sum_{i\in \beta_m}\left(\alpha_m(G_i)+\sum_{t\in H_m(G_i)}f_i(t)\right):f_i\in G_i \text{ for every }i\in \N \right)\right\rbrace\in A.$$
        \end{enumerate}
     \item There exist  idempotents $p_1,\ldots ,p_m$ in $(\mathcal{J},+)$, such that $A\in \tilde{h}\left(\bigotimes_{i=1}^m p_i\right).$

\end{enumerate}  
\end{cor}

In the above theorems, we did not consider any polynomial inside $h.$ This inspires us the following generalization of the above results.
\begin{cor}\label{todo5}
     Let $m\in \N$, $h:\N^m\rightarrow \N$ be any polynomial having no constant term, and let $p_1,\ldots ,p_m$ be minimal idempotents in $(\beta\mathbb N,+)$. Then for every $A\in \tilde{h}\left(\bigotimes_{i=1}^m p_i\right)$, and $T\in\mathcal P_f(P)$, we have two functions $\alpha_1,\ldots ,\alpha_m:\mathcal P_f(^\N\N)\rightarrow \mathbb N$, and $H_1,\ldots ,H_m:\mathcal P_f(^\N\N)\rightarrow \mathcal P_f(\mathbb N)$ such that

 \begin{enumerate}
     \item if $G,K\in\mathcal{P}_f(^\N\N)$, and $G\subsetneq K$, then for each $i\in \{1,\ldots ,m\}$, we have  $H_i(G)< H_i(K)$, and

     \item if $G_1\subsetneq G_2\subsetneq\dots\subsetneq G_n\subsetneq\cdots$ in $P_f(^\N\N)$, then for each $\beta_1<\beta_2<\dots<\beta_m,$ in $\G(\N)$,  $ f_i\in G_i$, and $P_1,\ldots ,P_m\in T$ we have 
     \[
     h\left(\sum_{i\in\beta_1} \alpha_1(G_i) + P_1\left(\sum_{i\in\beta_1} \sum_{t\in H_1(a_i)} f_i(t)\right),\ldots, \sum_{i\in\beta_m} \alpha_m (G_i)+P_m\left(\sum_{i\in\beta_m}\sum_{t\in H_m(G_i)}f_i(t)\right)\right) \in A.
     \]
     \item if for each $i\in \{1,\ldots ,m\}$, $p_i=p$ for some minimal idempotent $p$, then we can consider $\alpha_i=\alpha$, and $H_i=H$ for some functions $\alpha :\mathcal P_f(^\N\N)\rightarrow \mathbb N$, and $H:\mathcal P_f(^\N\N)\rightarrow \mathcal P_f(\mathbb N)$.
 \end{enumerate}
 In fact we may consider $p\in E\left(\mathcal{J}_p,+\right).$
\end{cor}
\begin{proof}
    Immediately follows from \cite[Theorem 3.2]{bhw}, and Corollaries \ref{todo3}, \ref{newmulpoly2}.
\end{proof}

\subsubsection*{Structure of the paper:} In the next section, we prove Theorem \ref{todo0}. Let $i:\N\rightarrow \N$ be a polynomial defined by $i(x)=x$ for all $x\in \N.$ Then if we choose $T=\{i\}$ in Theorem \ref{todo2}, then a similar proof of Theorem \ref{todo2} implies $(2)\implies (1)$ in Theorem \ref{todo0}. Hence we need to prove the reverse direction. So our first focus will be to prove $(1)\implies (2)$ in Theorem \ref{todo0}.  Similarly, in light of Theorem \ref{newmulpoly1}, we need to prove Theorem \ref{newmul1}, we need to show only $(1)\implies (2).$  
Then we prove Theorem \ref{todo2} and its another version Theorem \ref{newmulpoly1}. In the proof of Theorem \ref{todo2}, and Theorem \ref{newmulpoly1}, we used the induction hypothesis. One can easily lift this induction hypothesis (due to the presence of idempotent ultrafilters) to conclude the Corollaries $(2)\implies (1)$ in \ref{todo1},  $(2)\implies (1)$ in \ref{newmul2}, \ref{todo3}, and \ref{newmulpoly2}. However the proofs of Corollaries $(1)\implies (2)$ in \ref{todo1},  $(1)\implies (2)$ in \ref{newmul2} are similar to the proof of  $(1)\implies (2)$ in Theorem \ref{todo0}, and $(1)\implies (2)$ in Theorem \ref{newmul1}.  Then in the last section we address the polynomial extension of separating the Milliken-Taylor theorem, originally proved in \cite{d}.

\section{Our Proofs}

As we promised before, in this section we will prove Theorem \ref{todo0}, and Theorem \ref{todo2}. Then we draw several new results using these theorems.

\begin{proof}[{\bf Proof of Theorem \ref{todo0}:}]
    As mentioned earlier we are going to proof only $(1) \implies (2).$\\
     Let $\alpha:\mathbb P_f(^\N\N)\rightarrow\mathbb N$ and $H:\mathbb P_f(^\N\N)\rightarrow\mathbb P_f(\mathbb N)$ be two functions such that $\max H(F)<\min  H(G)$ provided $F,G\in\mathbb P_f(^\N\N)$ and $F\subsetneq G$. For $G\in\mathbb P_f(^\N\N)$, let $$T_G(\alpha, H)=\left\lbrace\sum_{i\in \beta}\left(\alpha(G_i)+\sum_{t\in H(G_i)}f_i(t)\right):n\in\mathbb N,\,\beta\subseteq [1,n], G\subsetneq G_1\subsetneq G_2\cdots\subsetneq G_n, \text{ and for each }i\in \{1,2,\ldots n\},f_i\in G_i\right\rbrace.$$
When $\alpha, H$ are fixed we write it $T_G$ instead of $T_G(\alpha, H).$

 Now for $\alpha:\mathbb P_f(^\N\N)\rightarrow\mathbb N$ and $H:\mathbb P_f(\mathbb N)\rightarrow\mathbb P_f(\mathbb N)$ with $\max H(F)<\min H(G)$ provided $F,G\in\mathbb P_f(^\N\N)$ and $F\subsetneq G$; define \[
 Q(\alpha, H)=\bigcap_{G\in \mathbb P_f(\mathbb N_{\mathbb N})}\overline{T}_G.
 \]
 It is easy to verify that $Q(\alpha,H)\neq\emptyset$ and is a subgroup of $(\beta\mathbb N,+)$. Furthermore, 
 $ K(Q(\alpha,H))\subseteq (\mathcal{J},+).$
We show, by downward induction on $l\in\{1,2,\ldots,m\}$ that for each $G\in\mathbb P_f(^\N\N)$,
     \begin{multline*}
    \left\{ \sum_{i\in\beta_l}(\alpha(G_i)+\sum_{t\in H(G_i)}f_i(t)), \sum_{i\in\beta_{l+1}}(\alpha(G_i)+\sum_{t\in H(G_i)}f_i(t)),\ldots,\sum_{i\in\beta_m}(\alpha(G_i)+\sum_{t\in H(G_i)}f_i(t)):\right.\\
    \left. f_i\in G_i,G\subseteq G_i\text{ for every }i\in\mathbb N\text{ and }\beta_l<\beta_{l+1}<\cdots<\beta_m \right\}\in \bigotimes_{i=l}^mp.
    \end{multline*}

   Let $p\in K\left(Q\left(\alpha ,H\right)\right).$ For $l=m$, we have $T_G\in p$ so that,
    \[
    \left\{\sum_{i\in\beta_m}(\alpha(G_i)+\sum_{t\in H(G_i)}f_i(t)): f_i\in G_i, G\subseteq G_i\text{ for every }i\in\mathbb N\right\}\subset \bigotimes_{i=m}^mp.
    \]
Therefore, let $l\in\{1,2,\ldots,m-1\}$ and assume that the statement is true for $l+1$. Now let, $G\in \mathbb P_f(^\N\N)$ and 
\begin{multline*}
    A=\left\{\left( \sum_{i\in\beta_l}(\alpha(G_i)+\sum_{t\in H(G_i)}f_i(t)), \sum_{i\in\beta_{l+1}}(\alpha(G_i)+\sum_{t\in H(G_i)}f_i(t)),\ldots,\sum_{i\in\beta_m}(\alpha(G_i)+\sum_{t\in H(G_i)}f_i(t))\right):\right.\\
    \left. f_i\in G_i,G\subseteq G_i\text{ for every }i\in\mathbb N\text{ and }\beta_l<\beta_{l+1}<\cdots<\beta_m \right\}.
    \end{multline*}

    We claim that
    \[
    T_G\subseteq \left\lbrace x\in\mathbb N:\left\lbrace (x_{l+1},x_{l+2},\ldots,x_m)\in\mathbb N^{m-l}:(x,x_{l+1},x_{l+2},\ldots,x_m)\in A\right\rbrace\in \bigotimes_{i=l+1}^mp\right\rbrace.
    \]
    To this end let $b\in T_G$. Pick $G_l\in\mathbb P_f(^\mathbb N{\mathbb N})$ with $G\subseteq G_l$ such that $b=\sum_{i\in\beta_l}\alpha(G_i)+\sum_{t\in H(G_i)}f_i(t)$. Let $F=\bigcup_{i\in\beta_l}G_i$. Then 
\begin{multline*}
    \left\{ \left( b,\sum_{i\in\beta_{l+1}}(\alpha(G_i)+\sum_{t\in H(G_i)}f_i(t)), \sum_{i\in\beta_{l+2}}(\alpha(G_i)+\sum_{t\in H(G_i)}f_i(t)),\ldots,\sum_{i\in\beta_m}(\alpha(G_i)+\sum_{t\in H(G_i)}f_i(t))\right):\right.\\
    \left. f_i\in G_i, F\subseteq G_i\text{ for every }i\in\mathbb N\text{ and }\beta_{l+1}<\beta_{l+2}<\cdots<\beta_m \right\}\subseteq A.
    \end{multline*}
So, $\{(x_{l+1},x_{l+2},\ldots,x_m)\in\mathbb N^{m-l}:(b,x_{l+1},x_{l+2},\ldots,x_m)\in A\}\in \bigotimes_{i=l+1}^mp$ as required.

\end{proof}

Now we are going to prove Theorem \ref{newmul1}.

\begin{proof}[{\bf Proof of Theorem \ref{newmul1}:}]
Similar to the previous discussion we need to prove only $(1)\implies (2).$ This proof is very similar to the above one. We need one technical change in the above proof:

 for each $j\in \{1,\ldots ,m\}$, let $\alpha_j:\mathbb P_f(^\N\N)\rightarrow\mathbb N$ and $H_j:\mathbb P_f(\mathbb N)\rightarrow\mathbb P_f(\mathbb N)$ with $\max H_j(F)<\min H_j(G)$ provided $F,G\in\mathbb P_f(^\N\N)$ and $F\subsetneq G$; define \[
 Q(\alpha_j, H_j)=\bigcap_{G\in \mathbb P_f(\mathbb N_{\mathbb N})}\overline{T_G(\alpha_j,H_j)}.
 \]
 Now proceed inductively as before. In the each step of the induction choose 
 $p_j\in  K(Q(\alpha_j,H_j))\subseteq (\mathcal{J},+).$
    
\end{proof}

Now we will prove Theorem \ref{todo2}, but before that let us recall one fact about ultrafilters from \cite{key-11}. We will apply this result. For any discrete semigroup $(S,\cdot )$, and  $p\in E\in \left(\beta S,\cdot \right)$, if $A\in p=p\cdot p$, then from \cite[Lemma 4.14]{key-11}, the set $A^\star =\{x\in A:-x+A\in p\}\in p.$

\begin{proof}[{\bf Proof of Theorem \ref{todo2}:}]
Let $p\in E (K(\beta\mathbb N,+))$, and $q= p\otimes p\otimes \dots \otimes p$, and let $A\in q$. By using Lifting lemma \ref{Lifting lemma}, pick for each $j\in \{1, \ldots, m\}$, some $D_j: \mathbb N^{j-1} \rightarrow \mathcal P(\mathbb N)$ such that 

\begin{enumerate}
    \item[($1$)] for each $j\in \{1,2,\ldots,m\}$, and for each $s\in\{1,2,\ldots,j-1\}, w_s\in D_s(w_1,w_2,\dots,w_{s-1})$, and then $D_j(w_1,\dots,w_{j-1})\in p$, and

    \item[($2$)] if for each $s\in\{1,2,\ldots,m\},w_s\in D_s(w_1,w_2,\ldots,w_{s-1})$, then $(w_1,w_2,\dots,w_m)\in A$.
\end{enumerate}

Now we proceed using the induction argument.\\

For the base case, let $K=\{f\}$. As $D_1(\emptyset)^\star\in p$, there exists $a\in\mathbb N$, and $\beta\in\mathcal P_f (\mathbb N)$ such that for all $P\in T, a+P\left(\sum_{t\in\beta}f(t)\right) \in D_1(\emptyset)^\star$. Define $\alpha(K)=a$, and $H(K)=\beta$. Now assume that $|K|> 1$, and that $\alpha(G)$, and $H(G)$ have been defined for all non-empty proper subsets $G$ of $K$ satisfying

\begin{enumerate}
    \item[($1$)] If $G, K\in\mathcal P_f(\mathbb N^{\mathbb N})$ and $G\subsetneq K$, then $H(G)<H(K)$.

    \item[($2$)] If $n\in\mathbb N$, $G_1\subsetneq G_2\subsetneq \ldots\subsetneq G_n\subsetneq K$, then for each $j\in\{1,2,\ldots,m\}$, and $\beta_1<\beta_2<\dots<\beta_j\le\{n\}$, $P_1,\ldots ,P_j\in T,$ and  
  \[
    \hspace{-7cm}\sum_{i\in\beta_j}\alpha(G_i)+P_j\left(\sum_{i\in\beta_j}\sum_{t\in H(G_i)}f_i(t)\right)\in
    \]
    \[
    D_j\left(\sum_{i\in\beta_1}\alpha(G_i)+P_1\left(\sum_{i\in\beta_1}\sum_{t\in H(G_i)}f_i(t)\right),\dots,\sum_{i\in\beta_{j-1}}\alpha(G_i)+P_{j-1}\left(\sum_{i\in\beta_{j-1}}\sum_{t\in H(G_i)}f_i(t)\right)\right)^\star.
\]
\end{enumerate}

Let, $R=\left\{\sum_{i\in\beta_j}\sum_{t\in H(G_i)}f_i(t):\beta_j\le\{n\}, f_i\in G_i\right\}$,
$F=\bigcup\{H(G):\emptyset\neq G\subsetneq K\}$, and let $m=\max F$ and, let for each $j\in \{1,2,\ldots,m\}$,
{\footnotesize
\[
M_j=\left\{\sum_{i\in\beta_j}\alpha(G_i)+P\left(\sum_{i\in\beta_j}\sum_{t\in H(G_i)}f_i(t)\right) :\emptyset\neq G_1\subsetneq\ldots\subsetneq G_n \subsetneq K,P\in T,
\text{
and for each } i\in\{1,\ldots,n\};f_i\in G_i\right\}.
\]
}
Then 
\begin{align*}
    &M_j\subseteq &\\
    &\bigcap_{\emptyset\neq G_1\subsetneq\ldots\subsetneq G_n \subsetneq K,P\in T} D_j\left(\sum_{i\in\beta_1}\alpha(G_i)+P_1\left(\sum_{i\in\beta_1}\sum_{t\in H(G_i)}f_i(t)\right),\dots,\sum_{i\in\beta_{j-1}}\alpha(G_i)+P_{j-1}\left(\sum_{i\in\beta_{j-1}}\sum_{t\in H(G_i)}f_i(t)\right)\right)^\star.&
\end{align*}

%$M_j\subseteq D_j^\star$.

Let $P\in T, d\in R$, define the polynomial $Q_{p,d}$ as $Q_{p,d}(y)=P(y+d)-P(d)$. Clearly $Q_{p,d}\in\mathbb P$. Let $S=T\cup \{Q_{p,d}:P\in T \text{ and } d\in R\}$.

Define

\begin{multline*}
E=\bigcap^m_{j=1}\left\{D_j\left(\sum_{i\in\beta_1}\alpha(G_i)+P_1\left(\sum_{i\in\beta_1}\sum_{t\in H(G_i)}f_i(t)\right),\ldots,\sum_{i\in\beta_{j-1}}\alpha(G_i)+P_{j-1}\left(\sum_{i\in\beta_{j-1}}\sum_{t\in H(G_i)}\right)\right)^\star:\right.\\
\left.\emptyset\neq G_1\subsetneq\ldots\subsetneq G_n \subsetneq K, P1,\ldots ,P_{j-1}\in T, \text{ and for each } i\in\{1,\dots,n\},f_i\in G_i\right\}
\end{multline*}

$$G=\bigcap_{\underset{j=1,\dots,m}{ 
x\in M;}}-x+D_j\left(\sum_{i\in\beta_1}\alpha(G_i)+P_1\left(\sum_{i\in\beta_1}\sum_{t\in H(G_i)}f_i(t)\right),\dots,\sum_{i\in\beta_{j-1}}\alpha(G_i)+P_{j-1}\left(\sum_{i\in\beta_{j-1}}\sum_{t\in H(G_i)}f_i(t)\right)\right)^\star\in p.$$

Now by \cite[Lemma 10]{ejc}, pick $a\in\mathbb N$, and $\beta\in\mathcal P_f(\mathbb N)$ such that $\min\beta>m$, and for all $Q\in S$, and all $f\in K$, 
\[
a+Q\left(\sum_{t\in\beta}f(t)\right) \in G.
\]
Let $\alpha(K)=a$, and $H(K)=\beta$.
Since $\min\beta> m$, $(1)$ is satisfied.
To verify ($2$), let $n\in \mathbb N$, and let $G_1,G_2,\ldots,G_n\in\mathcal P_f(\mathbb N_{\mathbb N})$ with $G_1\subsetneq G_2\subsetneq\ldots\subsetneq G_n=K$, for each $i\in{1,\ldots,n}$, let $f_i\in G_i$, and let $P\in T$.
If $j=1$, and $\beta_1<\{n\}$, then this follows from induction hypothesis. If $j=1$, and $\beta_1=\{n\}$, then $G_n=K$.

Now for all $P\in T,\alpha(K)+P\left(\sum_{t\in H(k)}f_i(t)\right)\in G\subset D_1^\star(\emptyset)$.
Otherwise if,
$\beta_1=\gamma\cup\{n\}$, then we can rename $\gamma$ as $\beta_1$, and $\{n\}$ as $\beta_2$.
Now we can proceed as follows. Choose any $\beta_1<\beta_2<\dots<\beta_j<\{n\}$, where $2\leq j\leq m$. Then

\begin{multline*}
    \sum_{i\in\beta_j}\alpha(G_i)+P\left(\sum_{i\in\beta_j}\sum_{i\in H(G_i)}f_i(t)\right)=a+\sum_{i\in\beta_{j-1}}\alpha(G_i)+P\left(\sum_{i\in\beta_{j-1}}\sum_{t\in H(G_i)}f_i(t)\right) +\\
P\left(\sum_{t\in\beta}f_n(t)+\sum_{i\in\beta_{j-1}}\sum_{t\in H(G_i)}f_i(t)\right)-P\left(\sum_{t\in\beta}f_n(t)+\sum_{i\in\beta_{j-1}}\sum_{t\in H(G_i)}f_i(t)\right)\\
    =a+x+Q_{p,d}\left(\sum_{t\in\beta}f_n(t)\right)\in D_j\left(\sum_{i\in\beta_1}\alpha(G_i)+P\left(\sum_{t\in\beta_1}\sum_{t\in H(G_i)}f_i(t)\right),\ldots,\sum_{i\in\beta_{j-1}}\alpha(G_i)+\right.\\
\left.P\left(\sum_{i\in\beta_{j-1}}\sum_{t\in H(G_i)}f_i(t)\right)\right),
\end{multline*}
where, 
$x=\sum_{i\in\beta_{j-1}}\alpha(G_i)+P\left(\sum_{i\in\beta_{j-1}}\sum_{t\in H(G_i)}f_i(t)\right),d=\sum_{i\in\beta_{j-1}}\sum_{t\in H(G_i)}f_i(t)$.
\end{proof}

Now we are going to prove Theorem \ref{newmulpoly1}.

\begin{proof}[{\bf Proof of Theorem \ref{newmulpoly1}:}] The proof is similar to the proof of Theorem \ref{todo2}.\\

Let $p_1,\ldots ,p_m\in E (K(\beta\mathbb N,+))$, and $q= \bigotimes_{j=1}^m p_j$, and let $A\in q$. By using Lifting lemma \ref{Lifting lemma}, pick for each $j\in \{1, \ldots, m\}$, some $D_j: \mathbb N^{j-1} \rightarrow \mathcal P(\mathbb N)$ such that 

\begin{enumerate}
    \item[($1$)] for each $j\in \{1,2,\ldots,m\}$, and for each $s\in\{1,2,\ldots,j-1\}, w_s\in D_s(w_1,w_2,\dots,w_{s-1})$, and then $D_j(w_1,\dots,w_{j-1})\in p_j$, and

    \item[($2$)] if for each $s\in\{1,2,\ldots,m\},w_s\in D_s(w_1,w_2,\ldots,w_{s-1})$, then $(w_1,w_2,\dots,w_m)\in A$.
\end{enumerate}

Now we proceed using the induction argument.\\

For the base case, let $K=\{f\}$. As $D_1(\emptyset)^\star\in p_1$, there exists $a_1\in\mathbb N$, and $\gamma_1\in\mathcal P_f (\mathbb N)$ such that for all $P_1\in T, a_1+P_1\left(\sum_{t\in\gamma_1}f(t)\right) \in D_1(\emptyset)^\star$. Define $\alpha_1(K)=a_1$, and $H_1(K)=\gamma_1$. Now $\bigcap_{P_1\in T}D_2\left(a_1+P_1\left(\sum_{t\in\gamma_1}f(t)\right)\right)^\star \in p_2.$ Again for $K=\{f\}$, there exists $a_2\in\mathbb N$, and $\gamma_2\in\mathcal P_f (\mathbb N)$ such that for all $P_2\in T,$ we have $a_2+P_2\left(\sum_{t\in\gamma_2}f(t)\right) \in \bigcap_{P_1\in T}D_2\left(a_1+P_1\left(\sum_{t\in\gamma_1}f(t)\right)\right)^\star.$ Define $\alpha_2(K)=a_2$, and $H_2(K)=\gamma_2$. Inductively for each $j\in \{1,\ldots ,m\},$ we have 

\begin{enumerate}
    \item[($1$)]  $a_j\in \N$, and $\gamma_j\in \G(\N)$ such that for all $P_j\in T,$ $$a_j+P_j\left(\sum_{t\in\gamma_2}f(t)\right) \in \bigcap_{P_1,\ldots ,P_j\in T} D_j\left(a_1+P_1\left(\sum_{t\in\gamma_1}f(t)\right),\ldots ,a_{j-1}+P_{j-1}\left(\sum_{t\in\gamma_{j-1}}f(t)\right)\right)^\star\in p_j,$$ and 
    
    \item[($2$)]  $\alpha_j(K)=a_j$, and $H_j(K)=\gamma_j$.
    
\end{enumerate}

Now assume that $|K|> 1$, and that for each $j\in \{1,\ldots ,m\},$  $\alpha_j(G)$, and $H_j(G)$ have been defined for all non-empty proper subsets $G$ of $K$ satisfying

\begin{enumerate}
    \item[($1$)] If $G, K\in\mathcal P_f(\mathbb N^{\mathbb N})$ and $G\subsetneq K$, then $H(G)<H(K)$.

    \item[($2$)] If $n\in\mathbb N$, $G_1\subsetneq G_2\subsetneq \ldots\subsetneq G_n\subsetneq K$, then for each $j\in\{1,2,\ldots,m\}$, and $\beta_1<\beta_2<\dots<\beta_j\le\{n\}$, $P_1,\ldots ,P_j\in T,$ and  
  \[
    \hspace{-7cm}\sum_{i\in\beta_j}\alpha_j(G_i)+P_j\left(\sum_{i\in\beta_j}\sum_{t\in H_j(G_i)}f_i(t)\right)\in
    \]
    \[
    D_j\left(\sum_{i\in\beta_1}\alpha_1(G_i)+P_1\left(\sum_{i\in\beta_1}\sum_{t\in H_1(G_i)}f_i(t)\right),\dots,\sum_{i\in\beta_{j-1}}\alpha_{j-1}(G_i)+P_{j-1}\left(\sum_{i\in\beta_{j-1}}\sum_{t\in H_{j-1}(G_i)}f_i(t)\right)\right)^\star.
\]
\end{enumerate}

for each $j\in \{1,\ldots ,m\},$ let $R_j=\left\{\sum_{i\in\beta_j}\sum_{t\in H_j(G_i)}f_i(t):\beta_j\le\{n\}, f_i\in G_i\right\}$,
$F=\bigcup\{H_j(G):\emptyset\neq G\subsetneq K\text{ and }j\in \{1,\ldots ,m\}\}$. Let $M=\max F$ and, let for each $j\in \{1,2,\ldots,m\}$,
{\footnotesize
\[
M_j=\left\{\sum_{i\in\beta_j}\alpha_j(G_i)+P\left(\sum_{i\in\beta_j}\sum_{t\in H_j(G_i)}f_i(t)\right) :\emptyset\neq G_1\subsetneq\ldots\subsetneq G_n \subsetneq K,P\in T,
\text{
and for each } i\in\{1,\ldots,n\};f_i\in G_i\right\}.
\]
}
Then
\begin{align*}
    &M_j\subseteq &\\
    &\bigcap_{\emptyset\neq G_1\subsetneq\ldots\subsetneq G_n \subsetneq K,P\in T} D_j\left(\sum_{i\in\beta_1}\alpha_1(G_i)+P_1\left(\sum_{i\in\beta_1}\sum_{t\in H_1(G_i)}f_i(t)\right),\dots,\sum_{i\in\beta_{j-1}}\alpha_{j-1}(G_i)+P_{j-1}\left(\sum_{i\in\beta_{j-1}}\sum_{t\in H_{j-1}(G_i)}f_i(t)\right)\right)^\star.& 
\end{align*}

Define $R=\bigcup_jR_j.$ Let $P\in T, d\in R$, define the polynomial $Q_{p,d}$ as $Q_{p,d}(y)=P(y+d)-P(d)$. Clearly $Q_{p,d}\in\mathbb P$. Let $S=T\bigcup \{Q_{p,d}:P\in T \text{ and } d\in R\}$.

For each $j\in \{1,\ldots ,m\},$ define

\begin{multline*}
E_j=\left\{D_j\left(\sum_{i\in\beta_1}\alpha(G_i)+P_1\left(\sum_{i\in\beta_1}\sum_{t\in H(G_i)}f_i(t)\right),\ldots,\sum_{i\in\beta_{j-1}}\alpha(G_i)+P_{j-1}\left(\sum_{i\in\beta_{j-1}}\sum_{t\in H(G_i)}\right)\right)^\star:\right.\\
\left.\emptyset\neq G_1\subsetneq\ldots\subsetneq G_n \subsetneq K, P_1,\ldots ,P_{j-1}\in T, \text{ and for each } i\in\{1,\dots,n\},f_i\in G_i\right\}
\end{multline*}

$$G_j=\bigcap_{ 
x\in M_j}-x+D_j\left(\sum_{i\in\beta_1}\alpha(G_i)+P_1\left(\sum_{i\in\beta_1}\sum_{t\in H(G_i)}f_i(t)\right),\dots,\sum_{i\in\beta_{j-1}}\alpha(G_i)+P_{j-1}\left(\sum_{i\in\beta_{j-1}}\sum_{t\in H(G_i)}f_i(t)\right)\right)^\star\in p_j.$$
%\textcolor{red}{To be corrected}
Now by \cite[Lemma 10]{ejc}, pick $b_1\in\mathbb N$, and $\delta_1\in\mathcal P_f(\mathbb N)$ such that $\min\delta_1>m$, and for all $Q\in S$, and all $f\in K$, 
\[
b_1+Q\left(\sum_{t\in\delta_1}f(t)\right) \in G_1.
\]
Let $\alpha_1(K)=b_1$, and $H_1(K)=\delta_1$.

Now $E_1\cap G_2\cap D_1\left(\left\{b_1+Q\left(\sum_{t\in\delta_1}f(t)\right):f\in K, Q\in S\right\}\right)^\star \in p_2.$
Now again by \cite[Lemma 10]{ejc}, pick $b_2\in\mathbb N$, and $\delta_2\in\mathcal P_f(\mathbb N)$ such that $\min\delta_2>m$, and for all $Q\in S$, and all $f\in K$, 
\[
b_2+Q\left(\sum_{t\in\delta_2}f(t)\right) \in E_1\bigcap G_2\bigcap D_1\left(\left\{b_1+Q\left(\sum_{t\in\delta_1}f(t)\right):f\in K, Q\in S\right\}\right)^\star.
\]
Define $\alpha_2(K)=b_2$, and $H_2(K)=\delta_2$. Now for each $j\in \{1,\ldots ,m\}$, inductively one can construct $\alpha_j(K)=b_j$, and $H_j(K)=\delta_j$ such that 

\[
b_j+Q\left(\sum_{t\in\delta_j}f(t)\right) \in  E_{j-1} \bigcap G_j\bigcap D_{j-1}\left(\left\{\left(b_1+Q\left(\sum_{t\in\delta_1}f(t)\right)\ldots ,  b_{j-1}+Q\left(\sum_{t\in\delta_{j-1}}f(t)\right) \right):f\in K, Q\in S\right\}\right)^\star.
\]
Now the verification of the induction hypothesis is similar to the last part of the proof of Theorem \ref{todo2}. So we omit the rest.

\end{proof}

Now we show that ultrafilters witnessing the conclusion of Theorem \ref{newmulpoly1} forms a left ideal of $\beta ((\N^m),+ )$ and $\beta ((\N^m),\cdot ).$ Unfortunately we can't prove the analogous result for Theorem \ref{todo2}. 
Define $$\mathcal{L}=\overline{\left\lbrace p\in \beta (\N^m):\forall A\in p, \,\, A \text{ satisfies the conclusion of Theorem  } \,\ref{newmulpoly1}\right\rbrace}.$$ 

%and 
%$$\mathcal{M}=\left\lbrace p\in \beta (\N^k):\forall A\in p, \,\, A \text{ satisfies the conclusion of Theorem  } \,\ref{todo2}\right\rbrace.$$

The following corollary  follows from Theorem \ref{newmulpoly1}.
\begin{cor}
    $\mathcal{L}$ is a left ideal of $(\beta (\N^m),\cdot ),$ and  $(\beta (\N^m),+ ).$
\end{cor}

\begin{proof}
%Let $(x_1,\ldots ,x_m)\in \N^m$, and let $B=(x_1,\ldots ,x_m) \cdot A.$
Let  $p\in \mathcal{L}$, and $q\in \beta (\N^k)$ and $B\in q\cdot p.$ Let $T\in\mathcal P_f(P)$ be given. Then there exists $(x_1,\ldots ,x_m)\in \N^m$ such that $A=(x_1,\ldots ,x_m)^{-1}B\in p.$ Consider the new set of polynomials $T'=\{\frac{1}{x}P:P\in T, \text{ and } x\in \{x_1,\ldots ,x_m\}\}.$
     Now from Theorem \ref{newmulpoly1} , there exist functions $\alpha_1,\ldots ,\alpha_m:\mathcal P_f(^\N\N)\rightarrow \mathbb N$, and $H_,\ldots ,H_m:\mathcal P_f(^\N\N)\rightarrow \mathcal P_f(\mathbb N)$ such that

 \begin{enumerate}
     \item if $G,K\in\mathcal P_f(^\N\N)$, and $G\subsetneq K$, then for every $i\in \{1,\ldots ,m\}$, we have $H_i(G)< H_i(K)$, and

     \item if $n (\ge m)\in\mathbb N$, $G_1\subsetneq G_2\subsetneq\dots\subsetneq G_n$, then for each $\beta_1<\beta_2<\dots<\beta_m\le\{n\},$ in $\G(\N)$, $ f_i\in G_i$, and $P_1,\ldots ,P_m\in T$ we have 
     \[
     \left(\sum_{i\in\beta_1}  \alpha_1(G_i) + \frac{1}{x_1} P_1\left(\sum_{i\in\beta_1} \sum_{t\in H_1(G_i)} f_i(t)\right),\ldots, \sum_{i\in\beta_m} \alpha_m (G_i)+ \frac{1}{x_m} P_m\left(\sum_{i\in\beta_m}\sum_{t\in H_m(G_i)}f_i(t)\right)\right) \in A.
     \]
 \end{enumerate}
 For each $i\in \{1,\ldots ,m\},$ define $\alpha'_i:\mathcal P_f(^\N\N)\rightarrow \mathbb N$ as $\alpha'_i=x_i\cdot \alpha_i$. Then the new set of functions $(\alpha'_i)_{i=1}^m$, and $(H_i)_{i=1}^m$ works for the set $B.$ Hence $\overline{B}\cap \mathcal{L}\neq \emptyset.$ This completes the proof.

  The proof of  $\mathcal{L}$ is a left ideal of $(\beta (\N^m),+ )$ is easy to derive so we omit the proof.
\end{proof}

The following theorem shows that the combinatorial consequence of our Corollary \ref{todo3} is equivalent to a version of the polynomial extension of \cite[Theorem 17.31]{key-11}. The proof uses simple color induction arguments.

\begin{thm}\label{equivalent1}
Let $T\in \G(\mathbb{P})$. Then the following two versions are equivalent.
 
     \begin{enumerate}
      \item  Let $r,m\in\mathbb N$, and let $[\N]^m=\bigcup_{i=1}^rC_i$. Then there exists functions $\alpha : \mathcal{P}_f(^\N\mathbb{N})\to \mathbb{N}$ and $H: \mathcal{P}_f(^\N\mathbb{N}) \to \mathcal{P}_f\left(\mathbb{N}\right)$ such that 
         \begin{enumerate}
             \item \label{1.41} if $F,G \in \mathcal{P}_f(\tau)$ and $F \subsetneq G$ then $\max H(F) < \min H(G)$, and 
            \item \label{1.42} whenever $(G_i)_{i} \in \mathcal{P}_f(\tau)$, is a sequence of functions, i.e. $G_1 \subsetneq G_2 \subsetneq \cdots \subsetneq G_n\subsetneq \cdots$, then for every $P_1,\ldots ,P_m\in T,$ $\beta_1<\cdots<\beta_m$ in $\G(\N)$, we have 
            $$\left\lbrace \left( \sum_{i\in\beta_1} \alpha(G_i) + P_1\left(\sum_{i\in\beta_1} \sum_{t\in H(G_i)} f_i(t)\right),\ldots, \sum_{i\in\beta_m} \alpha (G_i)+P_m\left(\sum_{i\in\beta_m}\sum_{t\in H(G_i)}f_i(t)\right) \right):\right.$$\\
            $$\hspace{4in} f_i\in G_i \text{ for every }i\in \N \Bigg\} \subset C_j$$
            for some $j\in \{1,2,\ldots ,r\}.$
        \end{enumerate}
    \item Let $r,m\in\mathbb N$, and let $a_1,\dots,a_m\in\mathbb N$. Then for any $\mathbb N=\bigcup^r_{i=1}D_i$, there exists functions $\alpha : \mathcal{P}_f(^\N\mathbb{N})\to \mathbb{N}$ and $H: \mathcal{P}_f(^\N\mathbb{N}) \to \mathcal{P}_f\left(\mathbb{N}\right)$ such that for every $P\in T,$  
     \begin{enumerate}
             \item \label{1.41} if $F,G \in \mathcal{P}_f(\tau)$ and $F \subsetneq G$ then $\max H(F) < \min H(G)$, and 
            \item whenever $(G_i)_{i} \in \mathcal{P}_f(\tau)$, is a sequence of functions, i.e. $G_1 \subsetneq G_2 \subsetneq \cdots \subsetneq G_n\subsetneq \cdots$, then for every $\beta_1<\cdots<\beta_m$ in $\G(\N)$ we have for every $f_i\in G_i,$ and $P_1,\ldots ,P_m\in T,$ 
            $$a_1\cdot \left(\sum_{i\in\beta_1} \alpha(G_i) + P_1\left(\sum_{i\in\beta_1} \sum_{t\in H(G_i)} f_i(t)\right)\right)+\cdots +a_m\cdot \left(\sum_{i\in\beta_m} \alpha(G_i) + P_m\left(\sum_{i\in\beta_m} \sum_{t\in H(G_i)} f_i(t)\right)\right)\in D_j$$
            for some $j\in \{1,2,\ldots ,r\}.$
        \end{enumerate}
\end{enumerate}

\end{thm}

\begin{proof}
$(1)\implies(2)$
Let $\mathbb N=\bigcup^r_{i=1}D_i$, and induce a coloring $\left[\mathbb N\right]^{(m)}=\bigcup^r_{i=1}C_i$ by defining $(z_1,z_2,\dots,z_m)\in C_i\Leftrightarrow a_1z_1+\dots+a_mz_m \in D_i$. Now (2) directly follows from (1).\\

$(2)\implies(1)$ Proof is similar.   
\end{proof}

The following theorem is a variation of Theorem \ref{equivalent1}. Proof is similar to the proof of above one and hence we omit it.
\begin{thm}\label{equivalent2}
Let $T\in \G(\mathbb{P})$. Then the following two versions are equivalent.
 
     \begin{enumerate}
      \item  Let $r,m\in\mathbb N$, and let $[\N]^m=\bigcup_{i=1}^rC_i$. Then there exists functions $\alpha_1,\ldots ,\alpha_m : \mathcal{P}_f(^\N\mathbb{N})\to \mathbb{N}$ and $H_1,\ldots ,H_m: \mathcal{P}_f(^\N\mathbb{N}) \to \mathcal{P}_f\left(\mathbb{N}\right)$ such that 
         \begin{enumerate}
             \item \label{1.41} if $F,G \in \mathcal{P}_f(\tau)$ and $F \subsetneq G$ then for each $i\in \{1,2,\ldots ,m\}$, we have $\max H_i(F) < \min H_i(G)$, and 
            \item \label{1.42} whenever $(G_i)_{i} \in \mathcal{P}_f(\tau)$, is a sequence of functions, i.e. $G_1 \subsetneq G_2 \subsetneq \cdots \subsetneq G_n\subsetneq \cdots$, then for every $P_1,\ldots ,P_m\in T,$ $\beta_1<\cdots<\beta_m$ in $\G(\N)$ we have 
            $$\left\lbrace \left( \sum_{i\in\beta_1} \alpha_1(G_i) + P_1\left(\sum_{i\in\beta_1} \sum_{t\in H_1(G_i)} f_i(t)\right),\ldots, \sum_{i\in\beta_m} \alpha_m (G_i)+P_m\left(\sum_{i\in\beta_m}\sum_{t\in H_m(G_i)}f_i(t)\right) \right):\right. $$
           $$ \left. \hspace{3in} f_i\in G_i \text{ for every }i\in \N \right\rbrace\subset C_j$$
            for some $j\in \{1,2,\ldots ,r\}$.
            
        \end{enumerate}
    \item Let $r,m\in\mathbb N$, and let $a_1,\dots,a_m\in\mathbb N$. Then for any $\mathbb N=\bigcup^r_{i=1}D_i$, there exists functions $\alpha_1,\ldots ,\alpha_m : \mathcal{P}_f(^\N\mathbb{N})\to \mathbb{N}$ and $H_1,\ldots ,H_m: \mathcal{P}_f(^\N\mathbb{N}) \to \mathcal{P}_f\left(\mathbb{N}\right)$ such that  such that for every $P\in T,$  
     \begin{enumerate}
             \item \label{1.41} if $F,G \in \mathcal{P}_f(\tau)$ and $F \subsetneq G$ then for each $i\in \{1,2,\ldots ,m\}$, we have $\max H_i(F) < \min H_i(G)$, and  
            \item whenever $(G_i)_{i} \in \mathcal{P}_f(\tau)$, is a sequence of functions, i.e. $G_1 \subsetneq G_2 \subsetneq \cdots \subsetneq G_n\subsetneq \cdots$, then for every $\beta_1<\cdots<\beta_m$ in $\G(\N)$ we have for every $f_i\in G_i,$ and $P_1,\ldots ,P_m\in T,$ 
            $$a_1\cdot \left(\sum_{i\in\beta_1} \alpha(G_i) + P_1\left(\sum_{i\in\beta_1} \sum_{t\in H_1(G_i)} f_i(t)\right)\right)+\cdots +a_m\cdot \left(\sum_{i\in\beta_m} \alpha(G_i) + P_m\left(\sum_{i\in\beta_m} \sum_{t\in H_m(G_i)} f_i(t)\right)\right)\in D_j$$
            for some $j\in \{1,2,\ldots ,r\}.$
        \end{enumerate}
\end{enumerate}

\end{thm}

In the following theorem, we explicitly deduce some ultrafilters witnessing the polynomial extension of \cite[Theorem 17.31]{key-11}. However, we don't know if the converse part of this theorem is true or not.
\begin{thm}\label{book}
 Let $m\in\mathbb N,$ $a_1,\dots,a_m\in\mathbb N$, $\,p\in E\left( K(\beta \mathbb N,+)\right)$, and $A\in a_1p+\dots+a_m p$. Then for every $T\in\mathcal P_f(\mathbb P)$, there exists functions $\alpha : \mathcal{P}_f(^\N\mathbb{N})\to \mathbb{N}$ and $H: \mathcal{P}_f(^\N\mathbb{N}) \to \mathcal{P}_f\left(\mathbb{N}\right)$ such that 
         \begin{enumerate}
             \item \label{1.41} if $F,G \in \mathcal{P}_f(\tau)$ and $F \subsetneq G$ then $\max H(F) < \min H(G)$, and 
            \item \label{1.42} whenever $(G_i)_{i} \in \mathcal{P}_f(\tau)$, is a sequence of functions, i.e. $G_1 \subsetneq G_2 \subsetneq \cdots \subsetneq G_n\subsetneq \cdots$, then for every $P_1,\ldots ,P_m\in T,$ $\beta_1<\cdots<\beta_m$ in $\G(\N)$, we have 
             $$a_1\cdot \left(\sum_{i\in\beta_1} \alpha(G_i) + P_1\left(\sum_{i\in\beta_1} \sum_{t\in H(G_i)} f_i(t)\right)\right)+\cdots +a_m\cdot \left(\sum_{i\in\beta_m} \alpha(G_i) + P_m\left(\sum_{i\in\beta_m} \sum_{t\in H(G_i)} f_i(t)\right)\right)\in A.$$
        \end{enumerate}
        In fact we may consider $p\in E\left(\mathcal{J}_p,+\right).$
\end{thm}

\begin{proof}

In this proof we do not use \cite[Theorem 3.2]{bhw}, but we use simple concept of $p-\lim.$
 Let $\Phi: \mathbb N^m\rightarrow \mathbb N$ be the map defined by $\Phi (x_1,x_2,\dots,x_m)=\sum^m_{i=1}a_ix_i$. Clearly, $\Phi$ is a continuous map, and so we may apply \cite[Theorem 3.49]{key-11}. Let $\widetilde{\Phi}:\beta(\mathbb N^m)\rightarrow\beta\mathbb N$ be the continuous extension of $\Phi$. Let $p\in E(K(\beta\mathbb N,+))$, and $\bigotimes_{i=1}^m p\in\beta(\mathbb N^m)$. Now, 

\begin{align*}
&\widetilde{\Phi}\left[\bigotimes_{i=1}^m p\right]
=\widetilde{\Phi}\left[p-\lim_{x_1}\cdots p-\lim_{x_m}(x_1,\dots,x_m)\right]&\\
&=p-\lim_{x_1}\cdots p-\lim_{x_m}\Phi(x_1,\dots,x_m)&\\
&=p-\lim_{x_1}\dots p-\lim_{x_m}\sum^m_{i=1}a_ix_i&\\
&=a_1\cdot p+\dots+a_m \cdot p.&
\end{align*}
Let $A\in a_1 p+\dots+a_m p=\widetilde{\Phi}\left[\bigotimes_{i=1}^m p\right]$. Hence there exists $B\in \bigotimes_{i=1}^m p$ such that $\Phi[B]\subseteq A$. Now from Corollary \ref{todo3}, our result follows.
 
\end{proof}    

The following theorem is a variation of Theorem \ref{book}.

\begin{thm}\label{book1}
 Let $m\in\mathbb N,$ $a_1,\dots,a_m\in\mathbb N$, $\,p_1,\ldots ,p_m\in E\left( K(\beta \mathbb N,+)\right)$, and $A\in a_1p_1+\dots+a_m p_m$. Then for every $T\in\mathcal P_f(\mathbb P)$, there exists functions $\alpha_1,\ldots ,\alpha_m : \mathcal{P}_f(^\N\mathbb{N})\to \mathbb{N}$ and $H_1,\ldots ,H_m: \mathcal{P}_f(^\N\mathbb{N}) \to \mathcal{P}_f\left(\mathbb{N}\right)$ such that 
         \begin{enumerate}
             \item \label{1.41} if $F,G \in \mathcal{P}_f(\tau)$ and $F \subsetneq G$ then for each $i\in \{1,2,\ldots ,m\}$, we have $\max H_i(F) < \min H_i(G)$, and 
            \item \label{1.42} whenever $(G_i)_{i} \in \mathcal{P}_f(\tau)$, is a sequence of functions, i.e. $G_1 \subsetneq G_2 \subsetneq \cdots \subsetneq G_n\subsetneq \cdots$, then for every $P_1,\ldots ,P_m\in T,$ $\beta_1<\cdots<\beta_m$ in $\G(\N)$, we have 
             $$a_1\cdot \left(\sum_{i\in\beta_1} \alpha_1(G_i) + P_1\left(\sum_{i\in\beta_1} \sum_{t\in H_1(G_i)} f_i(t)\right)\right)+\cdots +a_m\cdot \left(\sum_{i\in\beta_m} \alpha_m(G_i) + P_m\left(\sum_{i\in\beta_m} \sum_{t\in H_m(G_i)} f_i(t)\right)\right)\in A.$$
        \end{enumerate}
        In fact we may consider $p_1,\ldots ,p_m\in E\left(\mathcal{J}_p,+\right).$
\end{thm}

\begin{proof}

The proof is similar to the proof of Theorem \ref{book}.
 Let $\Phi: \mathbb N^m\rightarrow \mathbb N$ be the map defined by $\Phi (x_1,x_2,\dots,x_m)=\sum^m_{i=1}a_ix_i$. Clearly, $\Phi$ is a continuous map, and so we may apply \cite[Theorem 3.49]{key-11}. Let $\widetilde{\Phi}:\beta(\mathbb N^m)\rightarrow\beta\mathbb N$ be the continuous extension of $\Phi$. Let $p_1,\ldots ,p_
m\in E(K(\beta\mathbb N,+))$, and $\bigotimes_{i=1}^m p_i\in\beta(\mathbb N^m)$. Now, 

\begin{align*}
&\widetilde{\Phi}\left[\bigotimes_{i=1}^m p_i\right]
=\widetilde{\Phi}\left[p_1-\lim_{x_1}\cdots p_m-\lim_{x_m}(x_1,\dots,x_m)\right]&\\
&=p_1-\lim_{x_1}\cdots p_m-\lim_{x_m}\Phi(x_1,\dots,x_m)&\\
&=p_1-\lim_{x_1}\dots p_m-\lim_{x_m}\sum^m_{i=1}a_ix_i&\\
&=a_1\cdot p_1+\dots+a_m \cdot p_m.&
\end{align*}
Let $A\in a_1 p_1+\dots+a_m p_m=\widetilde{\Phi}\left[\bigotimes_{i=1}^m p_i\right]$. Hence there exists $B\in \bigotimes_{i=1}^m p_i$ such that $\Phi[B]\subseteq A$. Now from Corollary \ref{newmulpoly2}, our result follows.
 
\end{proof}

The following corollary is an immediate consequence of Theorem \ref{book}.

\begin{cor}\label{now}
    Let $m\in\mathbb N,$ $h:\N^m\rightarrow \N$ be any polynomial with no constant term. Let $a_1,\dots,a_m\in\mathbb N$, $\,p\in E\left( K(\beta \mathbb N,+)\right)$, and $A\in \tilde{h}(a_1p+\dots+a_m p)$. Then for every $T\in\mathcal P_f(\mathbb P)$, there exists functions $\alpha : \mathcal{P}_f(^\N\mathbb{N})\to \mathbb{N}$ and $H: \mathcal{P}_f(^\N\mathbb{N}) \to \mathcal{P}_f\left(\mathbb{N}\right)$ such that 
         \begin{enumerate}
             \item \label{1.41} if $F,G \in \mathcal{P}_f(\tau)$ and $F \subsetneq G$ then $\max H(F) < \min H(G)$, and 
            \item \label{1.42} whenever $(G_i)_{i} \in \mathcal{P}_f(\tau)$, is a sequence of functions, i.e. $G_1 \subsetneq G_2 \subsetneq \cdots \subsetneq G_n\subsetneq \cdots$, then for every $P_1,\ldots ,P_m\in T,$ $\beta_1<\cdots<\beta_m$ in $\G(\N)$, we have 
             $$h\left(a_1\cdot \left(\sum_{i\in\beta_1} \alpha(G_i) + P_1\left(\sum_{i\in\beta_1} \sum_{t\in H(G_i)} f_i(t)\right)\right)+\cdots +a_m\cdot \left(\sum_{i\in\beta_m} \alpha(G_i) + P_m\left(\sum_{i\in\beta_m} \sum_{t\in H(G_i)} f_i(t)\right)\right)\right)\in A.$$
        \end{enumerate}
        In fact we may consider $p\in E\left(\mathcal{J}_p,+\right).$
\end{cor}
\begin{proof}
    Let $A\in \tilde{h}(a_1p+\dots+a_m p)$, and let $B\in a_1p+\dots+a_m p $ be such that $h[B]\subseteq A.$ Now Theorem \ref{book} directly implies our result.
\end{proof}

The following corollary is an immediate variation of Corollary \ref{now}.

\begin{cor}\label{now1}
    Let $m\in\mathbb N,$ $h:\N^m\rightarrow \N$ be any polynomial with no constant term. Let $a_1,\dots,a_m\in\mathbb N$, $\,p_1,\ldots ,p_m\in E\left( K(\beta \mathbb N,+)\right)$, and $A\in \tilde{h}(a_1p_1+\dots+a_m p_m)$. Then for every $T\in\mathcal P_f(\mathbb P)$, there exists functions $\alpha_1,\ldots ,\alpha_m : \mathcal{P}_f(^\N\mathbb{N})\to \mathbb{N}$ and $H_1,\ldots ,H_m: \mathcal{P}_f(^\N\mathbb{N}) \to \mathcal{P}_f\left(\mathbb{N}\right)$ such that 
         \begin{enumerate}
             \item \label{1.41} if $F,G \in \mathcal{P}_f(\tau)$ and $F \subsetneq G$ then for every $i\in \{1,\ldots ,m\}$, we have $\max H_i(F) < \min H_i(G)$, and 
            \item \label{1.42} whenever $(G_i)_{i} \in \mathcal{P}_f(\tau)$, is a sequence of functions, i.e. $G_1 \subsetneq G_2 \subsetneq \cdots \subsetneq G_n\subsetneq \cdots$, then for every $P_1,\ldots ,P_m\in T,$ $\beta_1<\cdots<\beta_m$ in $\G(\N)$, we have 
             $$h\left(a_1\cdot \left(\sum_{i\in\beta_1} \alpha_1(G_i) + P_1\left(\sum_{i\in\beta_1} \sum_{t\in H_1(G_i)} f_i(t)\right)\right)+\cdots +a_m\cdot \left(\sum_{i\in\beta_m} \alpha_m(G_i) + P_m\left(\sum_{i\in\beta_m} \sum_{t\in H_m(G_i)} f_i(t)\right)\right)\right)\in A.$$
        \end{enumerate}
        In fact we may consider $p_1,\ldots ,p_m\in E\left(\mathcal{J}_p,+\right).$
\end{cor}
\begin{proof}
    Let $A\in \tilde{h}(a_1p_1+\dots+a_m p_m)$, and let $B\in a_1p_1+\dots+a_m p_m $ be such that $h[B]\subseteq A.$ Now Theorem \ref{book1} directly implies our result.
\end{proof}

Though we don't know the ultrafilter characterization of the Theorems \ref{book}, \ref{book1}, but we can completely characterize them for the set $T=\{i\}$, where $i(x)=x$ for all $x\in \N.$ The following theorem is a variation of  \cite[Theorem 17.31]{key-11} for idempotents of $(\mathcal{J},+)$.

\begin{thm}\label{charac} 
Let $T=\{i\}\in \G(\mathbb{P})$, and $A\subseteq \N$.
\begin{enumerate}
    \item if $A$ satisfies the conclusion of Theorem \ref{book} if and only if there exists idempotent ultrafilter $p\in (\mathcal{J},+)$ such that $A\in a_1p+\cdot +a_mp.$

    \item if $A$ satisfies the conclusion of Theorem \ref{book1} if and only if there exists idempotent ultrafilters $p_1,\ldots ,p_m\in (\mathcal{J},+)$ such that $A\in a_1p_1+\cdot +a_mp_m.$
\end{enumerate}
    
\end{thm}

\begin{proof}
    The proof of the sufficient condition of both $(1)$ and $(2)$ is similar to the proof of Theorem \ref{book}, \ref{book1}. For the necessary part, both proofs are similar. So we prove the necessary part of $(2).$ 

    Let $B\subseteq \N^m$ be such that $(x_1,\ldots ,x_m)\in B\text{ if } \sum_{i=1}^ma_ix_i\in A.$  Then the set $B$ satisfies the combinatorial conclusion of the Corollary \ref{newmul2}. Hence there exist idempotents $p_1,\ldots ,p_m$ in $(\mathcal{J},+)$, such that $B\in \bigotimes_{i=1}^m p_i.$ Letting $\Phi: \mathbb N^m\rightarrow \mathbb N$ be the map defined by $\Phi (x_1,x_2,\dots,x_m)=\sum^m_{i=1}a_ix_i$, from Theorem \ref{book1}, we have $$\Phi[B]\in \tilde{\Phi}\left(\bigotimes_{i=1}^m p_i\right)=a_1p_1+\cdots+a_mp_m.$$
    But from the construction $\Phi[B]\subseteq A.$ Hence $A\in a_1p_1+\cdots+a_mp_m.$ This completes our proof.
\end{proof}

\noindent
\section{Separation of the Polynomial Milliken-Taylor Systems}

In \cite{d,h}, authors extensively studied and characterized the compatibility of two Milliken-Taylor systems over the set of integers. That means they addressed when two Milliken-Taylor systems generated by two distinct finite sets lie in the same color after a finite coloring of $\N.$ To address their result of \cite{d}, we need the notions of {\it compressed sequence.}

\begin{defn}\textcolor{white}{}
    \begin{enumerate}
        \item Let $S=\{\overrightarrow{a}:\overrightarrow{a}\text{ is a finite sequence in }\N\}.$
        \item The function $c:S\rightarrow S$ deletes any consecutive repeated terms.
        \item An element $\overrightarrow{a}\in S$ is compressed if and only if $\overrightarrow{a}=c(\overrightarrow{a}).$
        \item An equivalence relation on $S$ is defined by $a\approx b$ if and only if there exists a rational $\alpha$ such that $\alpha \cdot c(\overrightarrow{a})=c(\overrightarrow{b}).$
    \end{enumerate}
\end{defn}

The following theorem gives a characterization of two distinct Milliken-Taylor systems lie in the same partition. 
\begin{thm}\label{sep}\cite[Theorem 3.2,3.3]{d} 
Let $\overrightarrow{a},\overrightarrow{b}\in S.$ Then the following are equivalent.
\begin{enumerate}
    \item For every finite coloring there exists two sequences $\langle x_n\rangle_{n}$ and $\langle y_n\rangle_{n}$ such that $MT \left( \overrightarrow{a},\langle x_n\rangle_n\right)\bigcup MT\left( \overrightarrow{b},\langle y_n\rangle_n \right)$ is monochromatic.
    \item $a\approx b.$
\end{enumerate}
    
\end{thm}

In \cite{h}, authors extended Theorem \ref{sep} over the set of integers $\Z.$ They proved Theorem \ref{sep} under the condition that $\overrightarrow{a},\overrightarrow{b}$ are two finite sets from $\Z\setminus \{0\}.$
From \cite{bhw} an immediate question appear if there exists any polynomial extension of Theorem \ref{sep}. Throughout this article we have seen the ultrafilter characterizations of polynomial patterns are rich in today's technique. But here we prove the polynomial extension of one side of Theorem \ref{sep}. Before that, we introduce a polynomial version of an equivalent compressed sequence.

\begin{defn}
    Let $m,n\in \mathbb N$, and $f:\mathbb N^m\rightarrow \mathbb N$, $g:\mathbb N^n\rightarrow \mathbb N$ be two polynomials of finite degree defined by $f(x_1,\dots,x_m)=\sum a_{i_1\cdots i_{m}} x^{i_1}_1\cdots x^{i_m}_m$, and $g(x_1,\dots,x_n)=\sum b_{j_1\cdots j_n}x^{j_1}_1\cdots x^{j_n}_n$. We say $f$ and $g$ are equivalent, i.e. $f\approx g$ if and only if,
    \begin{enumerate}
        \item both $f,g$ have same degree, and
        \item there exists a rational $\alpha\in \mathbb Q$ such that,
\[b_{i_1\cdots i_m}=\alpha^{i_1+\dots+i_m}a_{i_1\dots i_m}\forall\,(i_1,\dots, i_m).
\]
\end{enumerate}
\end{defn}

For example the polynomials $f(x,y,z)=x^3+2xy^2z+xz+z^2$, and $g(x,y,z)=5^3x^3+12\cdot 5^4 xy^2z+5^2xz+5^2z^2$ are equivalent.

Now we introduce the notion of the polynomial Milliken-Taylor system. This notion generalizes the notion of the Milliken-Taylor system \ref{defmillikentaylor}.

\begin{defn}[Polynomial Milliken-Taylor system]

Let $f:\mathbb N^m\rightarrow\mathbb N$ be any polynomial with no constant term. Let $\langle x_n\rangle^\infty_{n=1}$ be any sequence in $\mathbb N$. Define
\[
MT\left(f,\langle x_n\rangle^\infty_{n=1}\right)=\left\{ f\left(\sum_{i\in\beta_1} x_i,\dots,\sum_{i\in\beta_m} x_i\right):\beta_1<\dots<\beta_m\right\}.
\]

\end{defn}

The following theorem shows that if two polynomials are equivalent then their Milliken-Taylor image for some sequences lie in the same color after partitioning $\N$ into finitely many pieces.

\begin{thm}[Separating Polynomial Milliken-Taylor theorem]

Let $r,m\in\mathbb N$, and $f,g:\mathbb N^m\rightarrow\mathbb N$ be two finite degree polynomials with no constant term such that $f\approx g$. Then if $\mathbb N=\bigcup^r_{i=1}A_i$ is any $r$-coloring, there are two sequences $\langle x_n\rangle^\infty_{n=1}$, and
$\langle y_n\rangle^\infty_{n=1}$ in $\N$, and $i\in\{1,\dots,r\}$ such that $MT(f,\langle x_n\rangle^\infty_{n=1})\bigcup MT(g,\langle y_n\rangle^\infty_{n=1})\subseteq A_i$.
    
\end{thm}

\begin{proof}

Let $\alpha=\frac{p}{q}$. Let $f,g$ be two polynomials defined by:
\[
f(x_1,\dots,x_m)=\sum a_{i_1\cdots i_m}
x_1^{i_1}\cdots x_m^{i_m},\]
and 
\[g(x_1,\dots,x_m)=\sum b_{j_1}\cdots b_{j_m}
x_1^{j_1}\cdots x_
m^{j_m},\] where
$b_{i_1\cdots i_m}=\alpha^{i_1+\cdots+i_m}a_{i_1\cdots i_m}$ for every $(i_1,\dots,i_m)\in(\mathbb N\cup\{0\})^m.$

Let $h_f:\mathbb N^m\rightarrow \mathbb N$ be a new polynomial defined by $h_f(x_1,\dots,x_m)=\sum c_{i_1\cdots i_m}x^{i_1}_1\cdots x^{i_m}_m$, where $c_{i_1\cdots i_m}=p^{i_1+\cdots +i_m}a_{i_1\cdots i_m}$ for each $(i_1,\dots,i_m)\in (\mathbb N \cup \{0\})^m$. 
 From Theorem \ref{mt2}, there exists  $i\in\{1,\dots,r\}$, and a sequence $\langle z_n\rangle^{\infty}_{n=1}$ in $\mathbb N$ s.t. $MT(h_f,\langle z_n\rangle^{\infty}_{n=1})\subseteq A_i$. For each $n\in\mathbb N$, define $x_n=p\cdot z_n$, and $y_n=q\cdot z_n$.
 Then, $MT(f,\langle x_n\rangle^
{\infty}_{n=1})=MT(g,\langle y_n\rangle^{\infty}_{n=1})=MT(h_f,\langle z_n\rangle^{\infty}_{n=1})$. This completes the proof.
\end{proof}

\section*{Acknowledgement}  The first named author is supported by NBHM postdoctoral fellowship with reference no: 0204/27/(27)/2023/R \& D-II/11927. We are also thankful to Amit Roy for his technical assistantship.


\begin{thebibliography}{10}

\bibitem{b} M. Beiglb\"{o}ck: A multidimensional central sets theorem, Combin. Probab. Comput. 15 (2006),
807-814.

%\bibitem{bh1} V. Bergelson, and N. Hindman: Nonmetrizable topological dynamics and
%Ramsey Theory, Trans. Amer. Math. Soc. 320 (1990), 293-320.

%\bibitem{bhnew} V. Bergelson, and N. Hindman: Quotient sets and density recurrent sets, Trans. Amer. Math. Soc. 364 (2012), 4495-4531. 

\bibitem{bhw} V. Bergelson,   N. Hindman, and K. Williams: Polynomial extensions of the Milliken-Taylor Theorem, Trans. Amer. Math. Soc. 366 (2014), 5727-5748.

%\bibitem{bjm} V. Bergelson, J. H. Johnson Jr., J. Moreira: New polynomial and multidimensional extensions of classical partition results, J. Combin. Theory Ser. A.  147, April 2017, Pages 119-154.

\bibitem{bl} V. Bergelson and A. Leibman: Polynomial extensions
of van der Waerden and Szemeredi theorems, J. Amer. Math. Soc. 9 (1996)
725--753.

\bibitem{dhs} D. De, N. Hindman, and D. Strauss: A new and stronger Central Sets Theorem , Fundamenta Mathematicae 199 (2008), 155-175. 

\bibitem{d}  W. Deuber, N. Hindman, I. Leader, and H. Lefmann: Infinite partition regular matrices, Combinatorica 15 (1995), 333-355.

\bibitem{f2} H. Furstenberg: Recurrence in Ergodic Theory and Combinatorial Number Theory, Princeton
University Press, Princeton, NJ, 1981.

\bibitem{ejc}  S. Goswami, L. Luperi Baglini, and S. K. Patra: Polynomial extension of the Stronger Central Sets
Theorem,  Electronic Journal of Combinatorics,  30 (4), (2023), P4.36.

\bibitem{gxy} R. L. Graham, B. L. Rothschild, and J. H. Spencer. Ramsey theory. John Wiley \& Sons, Inc.,
New York, second edition, 1990.

\bibitem{key-6} N. Hindman: Finite sums from sequences within cells
of partitions of $\mathbb{N}$, J. Combin. Theory Ser. A, $17\, (1974),\,
1-11.$

\bibitem{new1} N. Hindman: A history of central sets, Ergodic Theory and Dynamical Systems 40 (2020), 1-33.

\bibitem{h} N. Hindman,  I. Leader, and D. Strauss: Separating Milliken-Taylor systems with negative entries, Proc. Edinburgh Math. Soc. 46 (2003), 45-61.

\bibitem{key-50} N. Hindman, and R. McCutcheon: Weak VIP systems in commutative semigroups, \emph{Topology Proc.}, 24: 199--221, 1999.

 \bibitem{key-51}  N. Hindman and R. McCutcheon: VIP systems in partial semigroups, Discrete Mathematics 240 (2001), 45-70.


\bibitem{key-11} N. Hindman, and D. Strauss: Algebra in the Stone-\v{C}ech
Compactifications: theory and applications, second edition, de Gruyter,
Berlin, $2012.$


\bibitem{tensor} N. Hindman, and  D. Strauss: Algebraic products of tensor products Semigroup Forum 103 (2021), 888-898.

\bibitem{}  N. Hindman, and  D. Strauss: Algebra in the Stone-Cech compactification -- an update Topology Proceedings 69 (2024), 1-69.

\bibitem{an} S. Kochen: Ultraproducts in the theory of models, Ann. of Math. 74 (1961),
221-261.

\bibitem{lupini} M. Lupini: Actions on semigroups and an infinitary Gowers–Hales–Jewett Ramsey theorem, Trans. Amer. Math. Soc., 371 (5), 3083-3116.

\bibitem{m} K. Milliken: Ramsey’s Theorem with sums or unions, J. Comb. Theory (Series A) $18$ $(1975)$,
$276$-$290.$

\bibitem{ramsey} F.P. Ramsey: On a problem of formal logic", Proceedings of the London Mathematical Society, 30: 264–286.

\bibitem{sc} I. Schur: \"{U}ber die Kongruenz $x^{m}+y^{m}\equiv z^{m}(\!\!\!\!\mod p)$,
Jahresbericht der Deutschen Math. -Verein., $25\, (1916),\, 114-117.$

\bibitem{new2}  H. Shi, and H. Yang: Nonmetrizable topological dynamical characterization
of central sets, Fund. Math. 150 (1996), 1-9.


\bibitem{t}  A. Taylor: A canonical partition relation for finite subsets of $\omega$, J. Comb. Theory (Series A)
$21 (1976), 137$-$146.$

\bibitem{key-8} B. van der Waerden: Beweis einer Baudetschen Vermutung,
Nieuw Arch. Wiskunde, $19\, (1927),\, 212-216.$

\bibitem{w}  K. Williams: Characterization of elements of polynomials in $\beta S$, Semigroup Forum, $83$, $147$–$160$, 
 $(2011)$.


\end{thebibliography}
\end{document}